\documentclass[11pt,english,french,leqno]{amsart}
\usepackage{amsmath,amssymb,stmaryrd,array,multirow,dsfont,mathrsfs}

\usepackage{cite}
\usepackage{graphics}
\numberwithin{equation}{section}
\usepackage{color}

\usepackage{hyperref}

\usepackage{bbm}
\DeclareMathAlphabet{\pazocal}{OMS}{zplm}{m}{n}
\usepackage{epsfig}
\usepackage{subfigure}

\renewcommand\d{\partial}
\def\eps{\varepsilon}
\DeclareMathOperator{\iD}{i}
\DeclareMathOperator{\eD}{e}
\DeclareMathOperator{\dif}{d}
\DeclareMathOperator{\tah}{th}
\DeclareMathOperator{\coh}{ch}
\DeclareMathOperator{\sih}{sh}
\DeclareMathOperator{\Span}{Span}
\DeclareMathOperator{\sign}{sgn}

\newcommand\R{\mathbb R}
\newcommand\C{\mathbb C}
\newcommand\Z{\mathbb Z}

\newcommand{\beqs}{\begin{equation*}}
\newcommand{\eeqs}{\end{equation*}}

\newcommand\br{\begin{remark}}
\newcommand\er{\end{remark}}
\newcommand\bp{\begin{pmatrix}}
\newcommand\ep{\end{pmatrix}}
\newcommand{\be}{\begin{equation}}
\newcommand{\ee}{\end{equation}}
\newcommand\ba{\begin{equation}\begin{aligned}}
\newcommand\ea{\end{aligned}\end{equation}}
\newcommand\ds{\displaystyle}

\newcommand{\bap}{\begin{app}}
\newcommand{\eap}{\end{app}}
\newcommand{\begs}{\begin{exams}}
\newcommand{\eegs}{\end{exams}}
\newcommand{\beg}{\begin{example}}
\newcommand{\eeg}{\end{exaplem}}
\newcommand{\bpr}{\begin {proposition}}
\newcommand{\epr}{\end{proposition}}
\newcommand{\bt}{\begin{theorem}}
\newcommand{\et}{\end{theorem}}
\newcommand{\bc}{\begin{corollary}}
\newcommand{\ec}{\end{corollary}}
\newcommand{\bl}{\begin{lemma}}
\newcommand{\el}{\end{lemma}}
\newcommand{\bd}{\begin{definition}}
\newcommand{\ed}{\end{definition}}
\newcommand{\brs}{\begin{remarks}}
\newcommand{\ers}{\end{remarks}}

\newtheorem{theorem}{Theorem}[section]
\newtheorem{proposition}[theorem]{Proposition}
\newtheorem{corollary}[theorem]{Corollary}
\newtheorem{lemma}[theorem]{Lemma}

\theoremstyle{remark}
\newtheorem{remark}[theorem]{Remark}
\theoremstyle{definition}
\newtheorem{definition}[theorem]{Definition}

\newtheorem{example}[theorem]{Example}

\newcommand\cB{{\mathcal B}}
\newcommand\cD{{\mathcal D}}
\newcommand\cU{\mathcal{U}}

\newcommand\cV{{\mathcal V}}

\newcommand\cC{{\mathcal C}}

\newcommand\cN{{\mathcal N}}

\newcommand\cF{{\mathcal F}}

\newcommand\cO{{\mathcal O}}

\newcommand\cM{{\mathcal M}}

\newcommand\cZ{{\mathcal Z}}

\newcommand\tF{{\tilde F}}
\newcommand\tG{{\tilde G}}
\newcommand\tq{{\tilde q}}

\title[Transverse instability of Stokes waves]{Small-amplitude finite-depth Stokes waves\\are transversally unstable}
\author{Ziang Jiao}
\address{Academy of Mathematics and Systems Science, Chinese Academy of Sciences, Beijing
100190 China.}
\email{jiaoziang@amss.ac.cn}
\thanks{}
\author{L.~Miguel Rodrigues}
\address{
Univ Rennes, CNRS, IRMAR - UMR 6625, F-35000 Rennes, France}
\email{{\tt luis-miguel.rodrigues@univ-rennes1.fr}}
\thanks{Research of L.M.R. was supported by the Institut Universitaire de France and partially supported by France 2030 through the programme Centre Henri Lebesgue ANR-11-LABX-0020-01.}
\author{Changzhen Sun}
\address{Universit\'e de Marie-et-Louis-Pasteur, CNRS, Laboratoire de Math\'ematiques de Besan\c{c}on (UMR CNRS 6623), F-25000 Besançon, France.}
\email{changzhen.sun@univ-fcomte.fr}

\author{Zhao Yang}
\address{State Key Laboratory of Mathematical Sciences, Academy of Mathematics and Systems Science, Chinese Academy of Sciences, Beijing
100190 China.}
\email{yangzhao@amss.ac.cn}
\thanks{}

\begin{document}

\maketitle

\begin{abstract}
We prove that all irrotational planar periodic traveling waves of sufficiently small-amplitude are spectrally unstable as solutions to three-dimensional inviscid finite-depth gravity water-waves equations.
\vspace{0.5em}

{\small \paragraph {\bf Keywords:} inviscid gravity water-waves equations; periodic traveling-waves; spectral instability; harmonic limit; Hamiltonian dynamics; Krein signature.
}

\vspace{0.5em}

{\small \paragraph {\bf 2020 MSC:} 76B15, 35B35, 35C07, 35B10, 37K45, 35Q35.
}
\end{abstract}

\section{Introduction}\label{sec:intro}

We contribute to the large-time stability analysis of the water-wave equations. Due to its mathematically challenging character and its practical impacts, the water-wave system has attracted so much attention that the literature devoted to the mathematical analysis of the water-wave system is overwhelmingly rich.  On its parts further from our main scope we simply refer the reader to\footnote{Obviously this does not mean that no significant improvement on those related fields have appeared since \cite{Lannes} but simply that we do not rely on those.} \cite{Lannes} for overall background. 

Our focus is on the dynamics near traveling waves. The literature devoted to traveling waves of water waves in their full generality (in type of flow and kind of wave) is also too large to be properly reported here and we refer instead the reader to the very recent review-and-perspective paper \cite{review}. As a piece of evidence that the field is not only large but also still fast growing, we point out that even since the very recent \cite{review}, significant progresses have been carried out, outside the main scope of the present paper, notably on infinite-depth Stokes waves \cite{Berti2022,creedon2023proof} and on capillary-gravity waves \cite{HTW,HY2023_cap}.

We restrict here our attention to equations describing a three-dimensional incompressible inviscid 
fluid over a flat bottom and under the influence of gravity and to their Stokes waves, that is, their curl-free planar periodic waves.
In particular the current contribution also belongs to the more general field studying periodic waves of Hamiltonian systems and the reader is referred to the introduction of\footnote{All along when quoting \cite{NRS-EP} we refer to the longer arXiv preprint {\tt arXiv:2210.10118} rather than to the shortened published version.} \cite{NRS-EP} for a small review and many references on the latter. Actually we point out that, to some extent,  along the years the possibility to cover the case of Stokes waves have served as an incentive for the development of a sufficiently general periodic-wave Hamiltonian theory. We refer to the cornerstone \cite{BM-BF} to exemplify how progresses in the study of Stokes waves have impacted periodic waves studies in general. We stress that to keep the discussion as focused as possible we restrict our discussion of the literature to its mathematical analysis side but we refer to the introduction of \cite{creedon2023proof} (devoted to the mathematical analysis of the infinite-depth case) as a good entering gate to the fluid mechanics, numerical and formal asymptotics sides. 

We defer to the next section reminders of how the water-wave equations and the corresponding Stokes waves are expressed in terms of mathematical symbols and instead provide here a rather informal statement of our main result. A more formal result is provided in Theorem~\ref{th:final}.

\begin{theorem}\label{th:main}
All finite-depth Stokes waves of sufficiently small amplitude are spectrally (and linearly\footnote{For the kind of exponential instabilities we consider the passage from spectral to linear is immediate. See Corollary~\ref{c:linear}.}) exponentially 
unstable under three-dimensional perturbations. 
\end{theorem}
 
We provide a complete proof that all small-amplitude finite-depth gravity Stokes waves are unstable as three-dimensional solutions. Yet we recall that, on one hand, at the spectral level, proving three-dimensional instability is a weaker\footnote{Roughly speaking because it involves a larger class of perturbations, the two-dimensional spectrum being included in the three-dimensional spectrum.} result than proving two-dimensional instability and, on the other hand, by combining \cite{BM-BF} with \cite{HY2023} one gets really close to a two-dimensional instability result. Concerning the former point, we point out that as customary in the field, we study an open-sea framework but we do not claim anything about propagation in channels\footnote{See the related discussion in \cite{YZ-multiD}, to be compared with \cite{JNRYZ,FRYZ}.}. 

Let us recall that when identifying waves coinciding up to translation and freezing the reference background fluid height, Stokes waves form a two-dimensional family, parameterized by spatial wavelength (or equivalently spatial wave number) and an amplitude-like parameter. In particular, untangling the zero-amplitude limit yields a one-dimensional family of harmonic profiles parameterized by wavelength. So far almost all the studies on the \textit{finite-depth} Stokes waves focus on their stability submitted to only longitudinal perturbations, that is, as a solution of two dimensional water wave system. 

The first mathematically rigorous study on the stability of finite-depth Stokes waves  goes back to \cite{BM-BF}, where the authors proved that there is a threshold (defined through a nonlinear equation) such that near harmonic waves whose wave number are above the threshold Stokes waves are spectrally unstable, validating the well-known  Benjamin-Feir instability criterion. The corresponding unstable spectrum is localized near the origin of the complex plane and its real parts grows as the square of the amplitude parameter, which is the maximal time-evolution growth scale. Though the Benjamin-Feir threshold is not completely explicit its defining equation is simple enough that even rigorous numerics, by interval arithmetics\footnote{See for instance \cite[Corollary~5.4]{HY2023}.}, may localize it fast and with good precision.

In the recent \cite{HY2023} (almost simultaneously with the formal expansions of \cite{Creedon_Deconinck_Trichtchenko_2022}) the other zone of the limiting spectrum from which an instability of maximal growth may arise was studied. There an index whose non-vanishing yields instability with growth scaling like the square of the amplitude was identified. Yet the supporting analysis involves computations carried out with symbolic softwares that are so heavy that it is dubious that they could be carried out by hand and the formula of the corresponding index is too large to be reasonably shown in a research paper. Needless to say that an analytical study of its vanishing seems out of reach. However a simple numerical evaluation of the index does hint at the fact that it vanishes only at one exceptional wave number, lying above the Benjamin-Feir threshold. What is missing to derive a complete proof of two-dimensional instability of all small-amplitude Stokes waves as conjectured in \cite{Creedon_Deconinck_Trichtchenko_2022,HY2023} is a proof that the index 
does not vanish below the Benjamin-Feir wave number threshold. For the moment the heaviness of the index formula precludes even a computer-assisted\footnote{As carried out in \cite{Barker} for the index of \cite{JNRZ_KdV-limit}.} proof. 

A few qualitative features of the instability proved in Theorem~\ref{th:main} are worth stating. 
\begin{enumerate}
\item As the two-dimensional instabilities studied in \cite{BM-BF,HY2023,BMV-BF}, the growth rate of the three-dimensional instability we prove scales at the maximal rate with respect to the wave amplitude, that is, as its square. For the sake of comparison we stress that in deep water two-dimensional perturbations of small-amplitude waves have already been proved to cause square-like growth rates --- see~\cite{Berti2022} --- so that a genuinely three-dimensional study would bring no further insight from this point of view.
\item The kind of perturbations we use is neither purely transverse, that is, not co-periodic in the longitudinal direction like in \cite{creedon2023proof} for infinite-depth Stokes waves,
nor two-dimensional, that is, not constant in the transverse direction. In this sense it strongly echoes the transverse instability of small-amplitude waves of the Schr\"odinger equations \cite{Audiard-Rodrigues}. 
\item As in \cite{BM-BF}, we do study an instability arising from the origin in the complex plane. Yet at a fixed amplitude the unstable spectrum we exhibit lies away from the origin (in a non-uniform way with respect to the amplitude) thus the exhibited instability is neither\footnote{For precise definitions of terminologies such as \emph{co-periodic}, \emph{modulational}, \emph{side-band}, \emph{etc.} we refer to \cite{Noble-Rodrigues}.} side-band, nor modulational. See Remark~\ref{rk:shape} for a detailed description of the shape of the spectrum studied to prove Theorem~\ref{th:main}. A similar remark applies to the two-dimensional instability proved in \cite{BMV-BF} for waves satisfying the Benjamin-Feir instability criterion.  By analogy with the Benjamin-Feir two-dimensional analysis \cite{BM-BF} and the Schr\"odinger analysis \cite{Audiard-Rodrigues} we expect that further work would also prove a modulational instability and as in \cite{Audiard-Rodrigues} connect it to the non hyperbolicity of an averaged\footnote{We refer to \cite{BGNR,BMR} for some background on the latter.} modulation system.  
\end{enumerate}

We would like now to say a few words about our strategy of proof. As a preliminary we point out that a significant part of our proof is inspired by \cite{NRS-EP}, on the spectral (in-)stability of electronic
Euler--Poisson system.  As acknowledged in its introduction, the analysis of \cite{NRS-EP} is itself partly inspired by formal expansions that have also been applied to study two-dimensional instabilities of Stokes waves \cite{Creedon_Deconinck_Trichtchenko_2022} away from zero in the complex planes, including those of \cite{HY2023}. We believe that the approach of \cite{Creedon_Deconinck_Trichtchenko_2022,NRS-EP} directly on the Bloch side is more tractable, direct and insightful\footnote{For instance a fact clear on numerical evaluations but not apparent in the analysis of \cite{HY2023} is that the index of \cite{HY2023} is always nonnegative hence an explanation for the numerically observed fact that the failure of its positivity is exceptional. In contrast, for the system studied in \cite{NRS-EP} the instability is concluded from the non-vanishing of some index, an event necessarily exceptional.} than the Evans' function approach\footnote{Incidentally we point out that the even more recent preprint \cite{berti2024high-frequency} does contain a Bloch-type analysis of the \cite{Creedon_Deconinck_Trichtchenko_2022,HY2023} zone, offering an alternative to the Evans' function approach of \cite{HY2023}.} of \cite{HY2023}; for a broader Bloch vs. Evans comparison we refer to \cite{R}. However, a fundamental additional ingredient is introduced here to identify an instability region that may be analyzed with a minimal level of computations so that they could be carried out with bare hands and yield an index readily evaluated through interval arithmetics so as to conclude instability in a final computer-assisted step of the proof. On this very last step we refer to \cite{Gomez-Serrano} for a brief introduction to the underlying techniques of computer-assisted proofs by interval arithmetics and a survey of some applications to free-boundary incompressible flows. 

To provide some further details, we recall (and anticipate on reminders of the next section) that the Zakharov/Craig-Sulem formulation transforms the original free-boundary three-dimensional differential systems of water-wave equations into pseudo-differential equations over $\R^2$ so that the generator of the linearized dynamics encoded by the latter equations is defined as acting on functions from $L^2(\R^2)$-based spaces. Let $L^\eps$ denote this operator, with $\eps$ marking the dependence on the amplitude of the background wave whereas the dependence on the wave number is left unmarked since it plays a passive role in the analysis. Then, by using a Fourier transform in the transverse direction and a Floquet-Bloch transform in the longitudinal direction the action of $L^\eps$ is decomposed as the actions of multipliers $L^\eps_{\ell,\xi}$ parameterized by a transverse frequency $\ell\in\R$ and a longitudinal Floquet exponent $\xi\in[-\pi/T,\pi/T)$, $T$ denoting the period of the background wave. Each multiplier $L^\eps_{\ell,\xi}$ is a pseudo-differential operator acting on functions from an $L^2(\R/(T\Z))$-based space and has compact resolvents hence discrete spectrum consisting entirely of eigenvalues of finite multiplicity.

We strive for the less computational situation to analyze. Because of the Hamiltonian structure an instability may arise only at non simple eigenvalues. Moreover $L^0_{\ell,\xi}$ being constant coefficients is diagonalized with $T$-periodic trigonometric monomials. From orthogonality of trigonometric monomials and the fact that in the
expansion of profiles as powers of $\eps$ the $\eps^m$-coefficient is a trigonometric polynomial of power less
than $m$, one deduces that the most computationally favorable case is the one where one finds a double eigenvalue of $L^0_{\ell,\xi}$ associated with eigenvectors of frequencies differing by twice the fundamental frequency, hence already connected by second order expansions of the profile. A less obvious, but well-known, fact is that the double eigenvalue must also exhibit non-definite Krein\footnote{We refer to the introduction of \cite{NRS-EP} for a more detailed discussion on the latter and to \cite[Chapter~7]{Kapitula-Promislow} for the related more basic background.} signature for an instability to arise.

The origin of the complex plane is a four-multiplicity eigenvalue of $L^0_{0,0}$ associated with eigenvectors with wave numbers either equal to $0$ or $\pm 2\pi/T$. First, by perturbing the latter we prove that there is a curve $\ell\mapsto (\ell,\xi_+(\ell))$ passing through $(0,0)$ such that when $\ell\neq0$, $L^0_{\ell,\xi_+(\ell)}$ possesses a double eigenvalue of opposite Krein signature, that is associated with eigenvectors with wave numbers $\pm 2\pi/T$, the other two small eigenvalues of $L^0_{\ell,\xi_+(\ell)}$ associated with constant eigenvectors being $\cO(\ell)$-away. Then, we analyze how this double eigenvalue of $L^0_{\ell,\xi_+(\ell)}$ splits under perturbation in $(\eps,\xi)$ by adapting the general framework designed in \cite{NRS-EP}. At this stage we obtain an index $\Gamma_\ell$ such that when it is nonzero then for some $c_\ell>0$ and any $\eps>0$ sufficiently small (depending on $\ell$) there is an open set $\Xi_{\ell,\eps}$ (shrinking to $\{\xi_+(\ell)\}$ when $\eps\to0$) such that when $\xi\in \Xi_{\ell,\eps}$, $L^\eps_{\ell,\xi}$ possesses an eigenvalue of real part bounded below by $c_{\ell}\,\eps^2$. 

Therefore, in the end,  to conclude the transverse instability one simply needs to check that the index $\Gamma_\ell$ is nonzero for some nonzero $\ell$. This is achieved by proving that $\ell\mapsto \Gamma_\ell$ possesses an analytic extension to a neighborhood of zero and computing explicitly the coefficients of its second-order Taylor expansion, which have a reasonably simple expression. Finally we appeal to interval arithmetics to prove that the three coefficients of the second-order expansion cannot vanish simultaneously. We stress that a large part of our analysis is devoted to going from an abstract formula for $\Gamma_\ell$ (given in Proposition~\ref{Prop:Gamma}) to an explicit Taylor expansion (whose coefficients are given in Corolllary~\ref{c:more} and Appendix~\ref{additional_formula}) suitable for arithmetic interval evaluations.

Let us add a final comment about how our proof deals with difficulties inherent to \emph{multidimensional} periodic-wave analyses. We first recall that the spectral analysis of analytic \emph{one-parameter} perturbations offers a framework that is considerably more regular and simpler than the general spectral perturbation theory and even than the \emph{two-parameters} perturbation theory as we consider here. Concerning the latter, we refer the reader to \cite{Katobook,Davies} for general background on spectral theory. Most transverse periodic stability analyses bypass this difficulty by restricting to co-periodic perturbations, which would correspond, here, to the study of how the spectrum of $L^\eps_{\ell,0}$ depends on $\ell$, or, more generally, to frozen Floquet perturbations. See for instance \cite{creedon2023proof} for the infnite-depth water-wave system and \cite{LBJM} for the Schr\"odinger equations. A notable exception is the Schr\"odinger analysis of \cite{Audiard-Rodrigues} and we point out that the full $(\ell,\xi)$ analysis was needed there to derive the spectral instabilty of small-amplitude waves. We stress that from being a purely technical question this lack of spectral regularity in multiple dimensions impacts the dynamics in various fundamental ways. It is responsible for dispersive effects of multidimensional hyperbolic systems or wave equations, enhancing local-in-space decay but deteriorating space localization. This affects periodic waves through their modulation systems;  see the detailed analysis in \cite{R2D2}. The spectral conical singularities also play a deep role in the dynamics of periodic Schr\"odinger operators; see for instance \cite{Fefferman-Weinstein,FKGH}. Our analysis fits in an intermediate place between one-parameter analyses and full multidimensional analyses in the sense that we do move both parameters $(\ell,\xi)$ but we do so in the neighborhood of a curve reaching the four-dimensional crossing with an asymptotic direction (thus not spiraling). This is the basic reason why we may benefit from both the freedom to move all parameters but maintain regularity, consistently with the general fact that regularity always holds in polar coordinates centered on the crossing.

\medskip

During the final phase of preparation of the present contribution the preprint \cite{CNS24} was released. Its goal is to extend \cite{creedon2023proof} from the deep-water case to the finite-depth case and in informal words its main upshot is that except possibly for a finite number of critical wavelengths small-amplitude Stokes waves are exponentially  unstable, a result overlapping with Theorem~\ref{th:main}. Yet the detailed spectral instability is dramatically different from the one we study here; see Theorem~\ref{th:final}. Indeed
\begin{enumerate}
\item The instability growth rate of \cite{CNS24} scales as the cube of the amplitude parameter.
\item It arises from a frozen Floquet analysis, thus through a one-parameter spectral study.
\end{enumerate}
The proof of \cite{CNS24} proceeds by deriving an instability index depending analytically on the spatial wavelength and checks that the index does not vanish in both the small wavelength and large wavelength regimes. For the reasons sketched above, to motivate our search for a double eigenvalue with wave numbers of eigenvectors differing by twice the fundamental wave number, the computations of \cite{CNS24} are so cumbersome that they are carried out with a symbolic software and result in an extremely heavy index. The comparison of \cite{creedon2023proof} with \cite{CNS24} strikingly illustrates that the finite-depth case is computationally much more demanding than the deep-water case. 

\medskip

Let us now list what we believe to be the most natural set of remaining open problems on the mathematical analysis of the finite-depth Stokes waves.
\begin{enumerate}
\item We believe that it would still be interesting to prove a two-dimensional instability, possibly but not necessarily by an interval arithmetic evaluation of the index of \cite{HY2023}.
\item The rigorous study of the spectral stability of large-amplitude Stokes waves is still largely open and any progress in this direction would be welcome.
\item It would be interesting to obtain a nonlinear instability. Here one may expect to extract some insights from the Benjamin-Feir deep-water nonlinear analysis of \cite{chen-su}. Let us however warn
the reader that as is unfortunately customary\footnote{The only exception we know about is \cite{DR}.} in the field the instability result proved there is of orbital type, and not of space-modulated type, whereas the latter is now known to be the sharp\footnote{See the parabolic nonlinear analysis in \cite{JNRZ-inventiones} and the linearized Hamiltonian analyses in \cite{R-linKDV,ARS}.} notion of stability for periodic waves. 
\end{enumerate}

\medskip

The rest of the paper is organized as follows. In the following section we gather some standard background, about the Zakharov/Craig-Sulem formulation, the existence of small-amplitude Stokes waves and our functional-analytic framework including considerations about Fourier/Floquet-Bloch symbols. Then we prove the existence of the aforementioned curve $(\ell,\xi_+(\ell))$ along which a double eigenvalue at zero-amplitude occurs. In the final and main section we derive the instability index $\Gamma_\ell$ and study its long-wave expansion.

\medskip

\noindent\emph{Acknowledgment.} L.M.R. and C.S. thank the Chinese Academy of Sciences and Z.Y. thanks the University of Rennes for the hospitality during the preparation of the present contribution.

\vspace{1em}

\noindent \emph{Data statement:} Data sharing not applicable to this article.

\noindent \emph{Conflict of interests:} None of the authors is in a situation which gives rise to a conflict of interest.

\vspace{1em}

\section{Preliminary background}

\subsection{The water wave equations}

In the original Eulerian formulation the water wave equations prescribe a dynamics along time $t$ for a fluid velocity $u(\cdot,t)\in\R^3$ defined on a free-boundary spatial domain 
\[
\cD_{\eta(\cdot,t)}:=\{(x,y,z)\in\R^3\,;\,0<z<\eta(x,y,t)\}
\]
and the surface elevation $\eta(\cdot,t)>0$ parameterizing the domain and defined on $\R^2$. Interior equations (in $\cD_{\eta(\cdot,t)}$) are the incompressible Euler equations with gravity force, whereas boundary conditions prescribe the non-penetration condition at the bottom $z=0$, that the fluid pressure equals the atmosphere pressure (assumed to be constant) at the top $z=\eta(x,y,t)$ and that the free-surface moves with fluid particles.

Restricting to curl-free\footnote{Proving instability for the curl-free system is a stronger result than proving it for the full system.} velocities, the no-penetration, curl-free and divergence-free conditions are equivalent to the existence of a $\phi(\cdot,t)$ such that $\phi(\cdot,t)$ is harmonic in $\cD_{\eta(\cdot,t)}$, satisfies the homogeneous Neumann condition at the bottom and $u(\cdot,t)=\nabla \phi(\cdot,t)$ with $\nabla=(\partial_x,\partial_y,\partial_z)$. Then the Euler equation is solved by prescribing the pressure field in terms of $\d_t\phi$ and $\nabla \phi$. Finally plugging the foregoing into the two boundary conditions at the top yields an evolution system for surface elevation $\eta$ and the trace $\psi$ at the surface of the velocity potential $\phi$, $\psi(\cdot,t):=\phi(\cdot,\eta(\cdot,t),t)$, known as the Zakharov/Craig-Sulem formulation. Explicitly it reads as 
\be \label{premain}
\left\{\begin{array}{ll}
    \d_t\eta(\cdot,t)&\ds=\ G[\eta(\cdot,t)]\psi(\cdot,t),\\[3pt]
\d_t\psi(\cdot,t)&\ds=\ -g(\eta(\cdot,t)-h)- \frac{1}{2}|\nabla \psi(\cdot,t)|^2+\frac{1}{2}\frac{\big(G[\eta(\cdot,t)]\psi(\cdot,t)+\nabla \eta(\cdot,t) \cdot \nabla \psi(\cdot,t)\big)^2}{1+|\nabla \eta(\cdot,t)|^2},
\end{array}\right.
\ee
where $\nabla=(\partial_x,\partial_y)$, $g>0$ is a gravity constant, $h>0$ is an arbitrary constant of integration and the Dirichlet-Neumann operator $G[\eta]$ is defined by  
\be \label{DN}
G[\eta(\cdot,t)]\psi(\cdot,t):=\big(\phi_z(\cdot,z,t)-\nabla \eta(\cdot,t) \cdot \nabla \phi(\cdot,z,t)\big)|_{z=\eta(\cdot,t)}
\ee
with $\phi$ the bounded\footnote{Some limitation on the growth at infinity must be imposed to ensure uniqueness.} solution to
\begin{align}\label{eqDN}
\Delta \phi(\cdot,t)&=0\ \textrm{ in }\cD_\eta\,,&
\d_z\phi(\cdot,0,t)&=0\,,&
\phi(\cdot,\eta(\cdot,t),t)&=\psi(\cdot,t)\,.
\end{align}

Note that some arbitrary choice has been made here, manifested by the presence of a reference height $h>0$. Indeed one could add any function of time to the velocity potential without changing the Eulerian solution. Our choice enforces that the steady constant Eulerian solution $(\eta,u)\equiv (h,0)$ corresponds to a steady constant solution to \eqref{premain}, namely $(\eta,\psi)\equiv(h,0)$. On a related note we observe that our instability result is sufficiently precise to track that it does correspond to a genuine Eulerian instability and not to a spurious instability due to the reformulation.

On properties of the Dirichlet-Neumann operator we refer to \cite[Chapter~3 and Appendix~A]{Lannes}. We only make use of properties that do not use spatial decay or periodicity in a fundamental way, unlike those used in low-regularity framework. Indeed the properties we use also hold in uniformly local Sobolev spaces, as sketched in \cite{Lannes} and proved in \cite{Alazard-Burq-Zuily_non-localized}. The regularity threshold for such robust framework is the one providing embeddings in Lipschitz\footnote{Actually even in $BUC^1$, the space of uniformly $C^1$ functions that are bounded and Lipschitz.} spaces. This ensures that the elliptic problem \eqref{eqDN} is set on a Lipschitz\footnote{Actually uniformly $\cC^1$.} domain that may be uniformly straightened out. To illustrate the properties we shall use note for instance, see \cite[Theorem~A.11]{Lannes}, that for any $s_0>2$, $0\leq s\leq s_0$ and $\psi\in H^s(\R^2)$ the map $G[\cdot]\psi$ is analytic from $h+H^{s_0}(\R^2)$ to $H^{s-1}(\R^2)$. Similar statements hold when one replaces $\R^2$ with $\R^2/\Lambda$, $\Lambda$ being a closed subgroup of $\R^2$, changes in the proof being mostly notational (requiring to replace some Fourier transforms with Fourier series or to omit some variables).

\subsection{Wave profiles}

A Stokes wave is a plane periodic wave of the water-wave equations, namely a solution to \eqref{premain} of the form $(x,y,t)\mapsto (\underline{\eta},\underline{\psi})(\underline{K}\cdot((x,y)-t\,\underline{c}))$ with wavevector $\underline{K}\in\R^2$, phase velocity\footnote{Only the longitudinal component of the velocity $\underline{K}\cdot\underline{c}$ is uniquely defined.} $\underline{c}\in\R^2$ and $2\pi$-periodic one-dimensional profile $(\underline{\eta},\underline{\psi})$.

It is convenient to use symmetries of the equations to reduce the number of parameters. Since the system is invariant by rotation in $(x,y)$ we may enforce $\underline{K}=\kappa\,(1,0)$ and $\underline{c}=c\,(1,0)$ with $\kappa>0$ and $c\geq0$. It is straightforward to rule out the possibility that $c=0$ for non constant waves so that we may also scale $(h,c,g)$ into a single parameter by replacing $(\eta,\psi,x,y,t)$ with $(\tilde{\eta},\tilde{\psi},\tilde{x},\tilde{y},\tilde{t})$ defined through $(\tilde{x},\tilde{y},\tilde{t}):=((x-c\,t)/h,y/h,c\,t/h)$ and
\begin{align*}
(\tilde{\eta},\tilde{\psi})(\tilde{x},\tilde{y},\tilde{t})
:=\left(\frac{\eta-h}{h},\frac{\psi}{c\,h}\right)\left(\frac{x-ct}{h},\frac{y}{h},\frac{c\,t}{h}\right)\,.
\end{align*}
Dropping tildes and denoting 
\begin{equation}\label{def:mu}
\mu:=\frac{gh}{c^2},
\end{equation}
the inverse of the Froude number we obtain
\ba \label{main-eq-dimensionless}
\left\{\begin{array}{ll}
\d_t\eta(\cdot,t)&\ds=\ \d_x\eta(\cdot,t)+G[1+\eta(\cdot,t)]\psi(\cdot,t)\,,\\[3pt]
\d_t\psi(\cdot,t)&\ds=\ \d_x\psi(\cdot,t)-\mu\,\eta(\cdot,t)- \frac{1}{2}|\nabla\psi(\cdot,t)|^2\\[3pt]
&\ds\qquad+\frac{1}{2}\frac{\left(G[1+\eta(\cdot,t)]\psi(\cdot,t)+\nabla\eta(\cdot,t)\cdot\nabla\psi(\cdot,t)\right)^2}{1+|\nabla\eta(\cdot,t)|^2}\,.
\end{array}\right.
\ea 
We stress that $\mu$ plays the role of a wave parameter (unlike $g$), replacing both $h$ and $c$. We warn the reader that if we were planning to perform a modulational analysis (as in \cite{Audiard-Rodrigues}) freezing the direction and blurring the genuine role of wave parameters would certainly be an inconvenient choice. 

Stokes waves are then obtained by determining $(\underline{\eta},\underline{\psi})$ periodic and one-dimensional and $\underline{\mu}>0$ solving
\ba \label{profiles}
\left\{\begin{array}{ll}
\underline{\eta}'+G[1]\underline{\psi}&\ds=\ -(G[1+\underline{\eta}]-G[1])\underline{\psi}\,,\\[3pt]
\underline{\psi}'-\underline{\mu}\,\underline{\eta}&\ds=\ \frac{1}{2}(\underline{\psi}')^2
-\frac{1}{2}\frac{\left(G[1+\underline{\eta}]\underline{\psi}+\underline{\eta'}\,\underline{\psi}'\,\right)^2}{1+(\underline{\eta}')^2}\,.
\end{array}\right.
\ea 
Since restricted to one-dimensional functions
\begin{align*}
G[1]&=D\tah(D)\,,&\textrm{with }D&:=-\iD\frac{\dif }{\dif x}\,,
\end{align*}
when seeking for solutions to \eqref{profiles} with $(\underline{\eta},\underline{\psi})$ small and $2\pi/\kappa$-periodic it is natural to search for $\underline{\mu}$ near 
\[
\mu_0\,:=\,\frac{\kappa}{\tah(\kappa)}\,.
\]
We point out that above and elsewhere we employ $\tah$, $\sih$ and $\coh$ to denote hyperbolic functions,
\begin{align*}
\sih(\cdot)&=\sinh(\cdot),&\coh(\cdot)&=\cosh(\cdot),&\tah(\cdot)&=\tanh(\cdot).
\end{align*}
Using the aforementionned regularity properties of the Dirichlet-Neumann operator transforms the proof of the existence of small-amplitude Stokes waves, stated below, into a simple instance of the Lyapunov-Schmidt method. We point out however that (a form of) the result was proved many decades before such tools were made available. Indeed it was first proved in \cite{Struik-exis-periodic} through a complex-analytic formulation \emph{\`a la }Levi-Civita. A Lyapunov-Schmidt-type proof is provided in \cite{BM-BF}.

For the sake of concision, we denote henceforth $H^s_{{\rm per},\kappa}:=H^s(\R/(2\pi\kappa^{-1}\Z))$.

\begin{proposition}[Stokes expansion]\label{p:Stokes}
Let $\kappa>0$ and $s_0>3/2$. There exist $\eps_0>0$, $K_0$ and an analytic map $\eps\mapsto (\underline{\eta}_\eps,\underline{\psi}_\eps,\underline{\mu}_\eps)$ from $(-\eps_0,\eps_0)$ to $H^{s_0}_{{\rm per},\kappa}\times H^{s_0}_{{\rm per},\kappa}\times \R_+^*$ such that for any $|\eps|<\eps_0$, $(\underline{\eta}_\eps,\underline{\psi}_\eps,\underline{\mu}_\eps)$ is the unique solution to \eqref{profiles} such that
\[
\|(\underline{\eta}_\eps,\underline{\psi}_\eps)-\eps\,(\sih(\kappa)\cos(\kappa\,\cdot\,),\coh(\kappa)\sin(\kappa\,\cdot\,))\|_{H^{s_0}_{{\rm per},\kappa}}+\left|\underline{\mu}_\eps-\frac{\kappa}{\tah(\kappa)}\right|
\leq K_0\,\eps^2\,.
\]
Moreover for $|\eps|<\eps_0$, $\underline{\eta}_\eps$ is even and $\underline{\psi}_\eps$ is odd and for any $s\geq0$,  the map $\eps\mapsto (\underline{\eta}_\eps,\underline{\psi}_\eps,\underline{\mu}_\eps)$ is also analytic from $(-\eps_0,\eps_0)$ to $H^s_{{\rm per},\kappa}\times H^s_{{\rm per},\kappa}\times \R_+^*$ with expansions starting as
\begin{align}\label{Stokes_expansion}
\underline{\eta}_{\eps}&=\eps\,\eta_1+\eps^2\,\eta_2+\cO(\eps^3)\,,& 
\underline{\psi}_{\eps}&=\eps\,\psi_1+\eps^2\,\psi_2+\cO(\eps^3)\,,&
\underline{\mu}_\eps&=\mu_0+\eps^2\,\mu_2+\cO(\eps^3)\,,
\end{align}
with
\ba \label{etapsi} 
 \eta_1(x)&=\sih(\kappa)\cos(\kappa x)\,,&\eta_2(x)&=-\frac{\kappa\tah(\kappa)}{4}+\frac{\kappa(2 + \coh(2\kappa))}{4\tah(\kappa)}\cos(2\kappa x)\,, \\
\psi_1(x)&= \coh(\kappa)\sin(\kappa x)\,, &\psi_2(x)&=\frac{\kappa(\coh(2\kappa) + \coh^2(2\kappa) + 1)}{ 4(\coh(2\kappa)-1)}
\sin(2\kappa x)\,,
\ea 
and
\begin{align}\label{muj}
\mu_0&=\frac{\kappa}{\tah(\kappa)}\,,&\mu_2&=-\frac{\kappa^3(15\sih^2 (\kappa)+ 16\sih^4 (\kappa)+ 8\sih^6(\kappa) + 9)}{8\coh(\kappa)\sih^3(\kappa) }\,.
\end{align}
\end{proposition}

The explicit computation of coefficients in Proposition~\ref{p:Stokes} (assuming the existence) preceded even the pioneering work \cite{Struik-exis-periodic} since it is essentially due to Stokes himself. It is relatively straightforward; see for instance \cite[Appendix~B]{BMV-BF}. 

Incidentally we point out that the Dirichlet-Neumann operator is also known to have similar smoothness properties on spaces of analytic functions (see \cite{Alazard-Burq-Zuily_analytic}) so that a similar result is expected to hold also in analytic spatial regularity. Details of the latter have been worked out in the deep-water case \cite{BMV_analytic}.

We do not mark the dependence on the wave number $\kappa$ since it plays a rather passive role in our analysis. 

\subsection{Fourier/Floquet-Bloch symbols}

\subsubsection*{Linear generator}

By using again the expansions of the Dirichlet-Neumann operator \cite[Theorem 3.2.1]{Lannes}, one obtains the linearization of \eqref{main-eq-dimensionless} with $\mu=\mu_\eps$ about $(\underline{\eta}_\eps,\underline{\psi}_\eps)$
\ba \label{linearization1}   
\left\{ \begin{array}{l}
 \d_t \eta=\d_x \big((1-\underline{V}_\eps)\eta\big)-G[1+\underline{\eta}_\eps](\underline{B}_\eps \eta)+G[1+\underline{\eta}_\eps]\psi, \\[5pt]
 \d_t \psi=-\big(\underline{\mu}_\eps+\underline{B}_\eps\,\underline{V}_\eps'\big)\eta-\underline{B}_\eps G[1+\underline{\eta}_\eps](\underline{B}_\eps\,\eta)
 +(1-\underline{V}_\eps)\d_x\psi+\underline{B}_\eps G[1+\underline{\eta}_\eps]\psi.
\end{array}\right.
\ea
where 
\begin{align} \label{shapederivative}
\underline{B}_\eps&:=\frac{G[1+\underline{\eta}_\eps]\underline{\psi}_\eps+\underline{\eta}_\eps'\,\underline{\psi}_\eps'}{1+(\underline{\eta}_\eps')^2}=\frac{\underline{\psi}_\eps'-1}{1+(\underline{\eta}_\eps')^2}\,\underline{\eta}_\eps'\,,&
\underline{V}_\eps&:=-\underline{B}_\eps\,\underline{\eta}_\eps'+\underline{\psi}_\eps'\,. 
\end{align}

At this stage, as in the local-wellposedness theory, in order to simplify \eqref{linearization1} and untangle its regularity structure it is expedient to introduce good unknowns \emph{à la }Alinhac, replacing $(\eta,\psi)$ with $U:=(\eta,\psi-\underline{B}_\eps\,\eta)$. Then system~\eqref{linearization1} becomes $(\d_t -L^{\eps})U=0$ with
\be \label{linearsys}
L^{\eps}= \bp
    \d_x \left((1-\underline{V}_\eps)\,\cdot\,\right)& G[1+\underline{\eta}_\eps] \\
    -\underline{\mu}_\eps+(1-\underline{V}_\eps)\underline{B}_\eps'&(1-\underline{V}_\eps)\,\d_x
\ep\,.
\ee
The linear part of the analysis in \cite[Chapter~4]{Lannes} shows that endowed with its maximal domain $L^\eps$ generates a strongly continuous semigroup of bounded operators on $H^s(\R^2)\times H^{s+\frac12}(\R^2)$ for any $s\geq 0$.

For the sake of concreteness we make the choice to consider $L^\eps$ as an operator on $L^2(\R^2)\times H^{\frac12}(\R^2)$ with domain $D_\eps$ thus given as the set of $U\in L^2(\R^2)\times H^{\frac12}(\R^2)$ such that $L^\eps U\,\in\,L^2(\R^2)\times H^{\frac12}(\R^2)$. Yet the instability we prove holds for any of the choices aforementionned. As expected the operator $L^\eps$ displays the Hamiltonian symmetry
\[
L^\eps\,=\,J\,A^\eps
\]
with 
\begin{align}\label{Ham}
J&:=\bp 0&-1\\1&0\ep\,,&
A^\eps&:=\bp
    -\underline{\mu}_\eps+(1-\underline{V}_\eps)\underline{B}_\eps'&(1-\underline{V}_\eps)\,\d_x\\
    -\d_x \left((1-\underline{V}_\eps)\,\cdot\,\right)& -G[1+\underline{\eta}_\eps]
\ep\,,
\end{align}
$J$ being invertible and skew-symmetric, $A^\eps$ being symmetric.

Providing a more explicit form of $D_\eps$ would be an unwieldy task, because of the anisotropic regularity structure of $L^\eps$. Fortunately making it more explicit would be of no particular help here.

Given the form \eqref{linearsys} it is expedient to introduce
\be \label{a}
\underline{a}_\eps:=\underline{\mu}_\eps-\mu_0-(1-\underline{V}_\eps)\underline{B}_\eps'
\ee
and convenient to extract from Proposition~\ref{p:Stokes} expansions for $(\underline{V}_\eps,\underline{a}_\eps)$.

\begin{lemma}\label{lem-expa}
With notation from Proposition~\ref{p:Stokes} and definitions~\eqref{shapederivative}-\eqref{a}, 
\begin{enumerate}
\item for any $|\eps|<\eps_0$, $(\underline{V}_\eps,\underline{a}_\eps)$ is an even function;
\item for any $s\geq0$, $\eps\mapsto (\underline{V}_\eps,\underline{a}_\eps)$ is an analytic function from $(-\eps_0,\eps_0)$ to $H^s_{{\rm per},\kappa}\times H^s_{{\rm per},\kappa}$ with expansions starting as
\begin{align*}
 \underline{V}_\eps&=\eps V_1+\eps^2 V_2+\cO(\eps^3)\,,&
 \underline{a}_\eps&=\eps a_1+\eps^2 a_2+\cO(\eps^3)\,,
\end{align*}
where
\begin{align*}
V_1(x)&=\kappa\coh(\kappa) \cos(\kappa x)\,,&
V_2(x)&=\kappa^2\frac{\sih^2(\kappa)}{2}+\kappa^2\frac{2 \sih^4(\kappa)+6 \sih^2(\kappa)+3}{4\sih^2(\kappa)}\cos(2\kappa x)\,,\\
a_1(x)&=-\kappa^2 \sih(\kappa) \cos(\kappa x)\,,&
a_2(x)&=-\kappa^3\frac{4\coh^6(\kappa)+3\coh^2(\kappa)+2}{8\coh(\kappa)\sih^3(\kappa)}-\kappa^3\frac{\coh^2(\kappa)+5}{2 \tah(\kappa)} \cos(2\kappa x)\,.
\end{align*}
\end{enumerate}
\end{lemma}

\subsubsection*{The Fourier/Floquet-Bloch transform}

Given the symmetries of the coefficient of $L^\eps$, it is compelling to introduce an integral transform that combines a Fourier transform in the $y$-direction and a Floquet-Bloch transform of period $T:=2\pi\kappa^{-1}$ in the $x$-direction. Explicitly, with\footnote{As usual for Fourier-type transforms formulas directly make sense for Schwartz functions, and are then extended by density and duality, using isometry on $L^2$-based spaces.} $g$ defined on $\R^2$ is associated $\cB(g)$ defined on $\R\times[-\kappa/2,\kappa/2)\times\R/T\Z$ by 
\[
\cB(g)(\ell,\xi,x)\ :=\ 
\sum_{j\in\Z} \eD^{\iD\,j\,\kappa\, x}\ \cF(g)(\xi+j\,\kappa,\ell)\,,
\]
where $\cF(g)$ is the Fourier transform normalized by
\[
\cF(g)(k,\ell)\ :=\ \frac{1}{(2\pi)^2} 
\int_{\R^2} \eD^{-\iD\,k\,x-\iD\ell\,y} g(x,y)\ \dif x\,\dif y\,,
\]
and one recovers $g$ from $\cB(g)$ through
\[
g(x,y)\ =\ \int_{-\frac{\kappa}{2}}^{\frac{\kappa}{2}}\int_{\R} \eD^{\iD\xi x+\iD\ell y}\ \cB(g)(\ell,\xi,x)\ \dif\ell\,\dif\xi\,.
\]
The transform $\cB$ diagonalizes $L^\eps$ in the sense that
\[
\cB(L^\eps U)(\ell,\xi,x)\ =\ L^\eps_{\ell,\xi}(\cB(U)(\ell,\xi,\cdot))(x)
\]
where each $L^\eps_{\ell,\xi}$ acts on functions from $L^2_{{\rm per},\kappa}\times H^{\frac12}_{{\rm per},\kappa}$ with domain $H^1_{{\rm per},\kappa}\times H^{\frac32}_{{\rm per},\kappa}$, through
 \be\label{Lell-xi}
L^{\eps}_{\ell,\xi}=\bp
(\d_x+i\xi)\left((1-\underline{V}_\eps)\,\cdot\,\right)& G_{\ell,\xi}[1+\underline{\eta}_\eps] \\
-(\mu_0+\underline{a}_\eps) &   (1-\underline{V}_\eps) (\d_x+i\xi)
\ep\,.
\ee 
In the foregoing, the Bloch symbol $G_{\ell,\xi}[1+\underline{\eta}_\eps]$ of $G[1+\underline{\eta}_\eps]$ may be suitably characterized as a Dirichlet-Neumann-like operator but for our purposes it is more convenient to use directly the abstract symbol formula 
\begin{align}\label{Gell-xi}
G_{\ell,\xi}[1+\underline{\eta}_\eps](\psi)(x)
\,=\,\eD^{-\iD \xi x-\iD \ell y}\,G[1+\underline{\eta}_\eps]\left((x',y')\mapsto \eD^{\iD \xi x'+\iD \ell y'}\psi(x')\right)(x,y)\,
\end{align}
where the formula is independent of the picked $y$ and $G[1+\underline{\eta}_\eps]$ is here considered as acting on uniformly local Sobolev spaces.
The latter point of view enables us to directly extract properties of $G_{\ell,\xi}[1+\underline{\eta}_\eps]$ from those known for $G[1+\underline{\eta}_\eps]$ seen as an operator on uniformly local Sobolev spaces, since uniformly over bounded sets of $(\xi,\ell)$, for any $s\geq 0$, the linear map
\begin{align*}
H^s_{{\rm per},\kappa}\to H^s_{{\rm ul}}(\R^2)\,,\qquad
\psi\longmapsto \left((x,y)\mapsto \eD^{\iD \xi x+\iD \ell y}\psi(x)\right)
\end{align*} 
is bounded with closed range and bounded inverse on this range. For instance this implies with bounds uniformly over bounded sets of $(\xi,\ell)$, for any nonnegative $s$, $s'$,
\[
\|G_{\ell,\xi}[1+\underline{\eta}_\eps]\|_{H^{s'}_{{\rm per},\kappa}\to H^s_{{\rm per},\kappa}}
\lesssim_{\ell,\xi} \|G[1+\underline{\eta}_\eps]\|_{H^{s'}_{{\rm ul}}(\R^2)\to H^s_{{\rm ul}}(\R^2)}\,.
\]

\subsubsection*{Expansions of the Dirichlet-Neumann symbols}

Combining Proposition~\ref{p:Stokes} with known\footnote{See \cite[Section~3.6.2]{Lannes}. A lower dimensional version is already present in \cite{CS93}. Since the uniformly local versions of results, that we are using here, are only sketched in \cite{Lannes}, let us offer here a quick walkthrough of the required adaptations. The key result to adapt is \cite[Proposition~3.28]{Lannes} and its uniformly local version is obtained starting from the $L^2$-type bound \cite[Corollary~2.46]{Lannes} (replacing  \cite[Corollary~2.40]{Lannes}) or the equivalent result in \cite{Alazard-Burq-Zuily_non-localized} and elaborating with product and commutator estimates of \cite[Appendix~B4]{Lannes}.} expansions of $G$ yields 
\begin{align}\label{expG}
G[1+\underline{\eta}_{\eps}]=G[1]+G^{(1)}[\underline{\eta}_{\eps}]+G^{(2)}[\underline{\eta}_\eps]+\cO(\eps^3),
\end{align}
where $G^{(1)}[\eta]$ is linear in $\eta$ and $G^{(2)}[\eta]$ is quadratic in $\eta$,
\begin{align*}
G[1]&=|D|\tah|D|\,,\\
G^{(1)}[\eta]&:=|D|(\eta-\tah(|D|)\eta \tah(|D|))|D|\,,\\
G^{(2)}[\eta]&:=-\frac{1}{2}|D|\Big(|D|\eta^2 \tah(|D|)+\tah(|D|)\eta^2 |D| -2\tah(|D|) \eta |D|\tah(|D|) \eta \tah(|D|)\Big)|D|\,,
\end{align*}
with $D:=-\iD\nabla$, and $\eta$ identified with the multiplication operator by $\eta$. In expansion~\ref{expG} the error is measured in operator norms and holds for any $s\geq0$ as operators from either $H^{s+1}(\R^2)$ to $H^s(\R^2)$, or from $H^{s+1}_{{\rm per},\kappa}$ to $H^s_{{\rm per},\kappa}$, or $H^{s+1}_{{\rm ul}}(\R^2)$ to $H^s_{{\rm ul}}(\R^2)$, which is the version we are actually using. Incidentally we point out that it is apparent on $G^{(1)}[\eta]$ that, as aforementioned, some regularity on $\eta$ is needed to be able to consider $G[\eta]$ as acting as a first-order operator, regular in $\eta$. Combining \eqref{Gell-xi} with \eqref{expG} (in its uniformly local version) and expansions of $\underline{\eta}_\eps$ yields the following expansions for $G_{\ell,\xi}[1+\underline{\eta}_\eps]$.

\begin{lemma}\label{lem-expG}
With notation from Proposition~\ref{p:Stokes}, for any $s\geq0$ as an operator from $H^{s+1}_{{\rm per},\kappa}$ to $H^s_{{\rm per},\kappa}$ the Bloch symbol $G_{\ell,\xi}[1+\underline{\eta}_\eps]$ is continuous in $\eps$ and expands as
\be\label{expan-DN-eps}
G_{\ell,\xi}[1+\underline{\eta}_\eps]= G_{\ell,\xi}[1]+\eps G_{\ell,\xi}^{(1)}+\eps^2G_{\ell,\xi}^{(2)}+\cO(\eps^3). 
\ee
where 
\begin{align}\label{expansionG}
G_{\ell,\xi}[1]&=|D|_{\ell,\xi}\tah(|D|_{\ell,\xi})\,,&
G_{\ell,\xi}^{(1)}&=\tG_{\ell,\xi}^{(1)}[\eta_1]\,,&
G_{\ell,\xi}^{(2)}&=\tG_{\ell,\xi}^{(1)}[\eta_2]+\tG_{\ell,\xi}^{(2)}[\eta_1]\,,&
\end{align}
with $\eta_1$ and $\eta_2$ as in \eqref{etapsi}, and
\begin{align*}
\tG_{\ell,\xi}^{(1)}[\eta]
&:=\,-G_{\ell,\xi}[1]\,\eta\, G_{\ell,\xi}[1]-(\d_x +\iD\xi)\,\eta \,(\d_x +\iD\xi)+\ell^2 \eta,\\
\tG_{\ell,\xi}^{(2)}[\eta]&:=\,
G_{\ell,\xi}[1]\,\eta\, G_{\ell,\xi}[1]\,\eta\, G_{\ell,\xi}[1]
-\frac12 |D|_{\ell,\xi}^2 \,\eta^2\, G_{\ell,\xi}[1]-\frac12 G_{\ell,\xi}[1]\,\eta^2\,|D|_{\ell,\xi}^2
\end{align*} 
$D=-\iD\frac{\dif }{\dif x}$, $|\cdot|_{\ell,\xi}:=\sqrt{(\cdot+\xi)^2+\ell^2}$, functions being identified with their associated multiplication operator.
\end{lemma}

Lemma~\ref{lem-expG} provides continuity\footnote{The proof of Lemma~\ref{lem-expG} provides higher order sharp regularity but it is not needed here.} in $\eps$, holding $(\ell,\xi)$ fixed, measured in operator norms with sharp regularity loss. In the following we shall also need joint higher regularity in $(\ell,\xi,\eps)$ but possibly with nonsharp regularity loss. The difference in requirements comes from the fact that in spectral perturbation one needs sharp norm control only when building spectral projector by contour integrals, since perturbation theory fundamentally hinges on norm continuity of rank of projectors. In turn, once spectral projectors are built, the construction of a reduced matrix describing the spectrum inside a fixed contour could start from a (very) smooth basis hence allow to absorb arbitrary finite regularity loss.

Such nonsharp regularity may be obtained, for instance, through conformal coordinates. The outcome of introducing conformal coordinates\footnote{The regularity of conformal mapping may also be deduced from the kind of shape regularity analysis of \cite{Lannes}. For instance one may introduce $\mathfrak{Z}_\eps$ the bounded solution to
\begin{align*}
\Delta \mathfrak{z}_\eps&=0\ \textrm{ in }
\cD^0_{1+\underline{\eta}_\eps}
:=\{(x,z)\in\R^2\,;\,0<z<1+\underline{\eta}_\eps(x)\}
\,,&
\mathfrak{z}_\eps(0)&=0\,,&
\mathfrak{z}_\eps(1+\underline{\eta}_\eps(x))&=1\,,
\end{align*}
and $\mathfrak{X}_\eps$ a harmonic conjugate of $\mathfrak{Z}_\eps$, that is, $\mathfrak{X}_\eps$ satisfies $\nabla\mathfrak{X}_\eps=\nabla^\perp\mathfrak{Z}_\eps$. Then $\mathfrak{h}_\eps$ is chosen to enforce the periodicity of $x\mapsto (1+\mathfrak{h}_\eps)\mathfrak{X}_\eps(x,y)-x$. At last one may set $\mathfrak{x}_\eps(x)=(1+\mathfrak{h}_\eps)\mathfrak{X}_\eps(x,1+\underline{\eta}_\eps(x))-x$.} is that for arbitrary large $s$, there exists an analytic map $(-\eps_0,\eps_0)\to H^s_{{\rm per},\kappa}\times H^s_{{\rm per},\kappa}\times H^s_{{\rm per},\kappa}\times \R$, $\eps\mapsto (\mathfrak{x}_\eps,\mathfrak{v}_\eps,\mathfrak{w}_\eps,\mathfrak{h}_\eps)$ starting from zero at $\eps=0$ such that 
\begin{align}
&G_{\ell,\xi}[1+\underline{\eta}_\eps](\psi)(x) \label{conformal}\\\nonumber
&\,=\,(1+\mathfrak{v}_\eps(x))\eD^{\iD \xi\,\mathfrak{x}_\eps(x)}\,
\left(G^{[\mathfrak{w}_\eps]}_{\ell,\xi}[1+\mathfrak{h}_\eps]\left(\left(x'\mapsto \eD^{-\iD \xi\,\mathfrak{x}_\eps(x')}\psi(x')\right)\circ(\textrm{Id}_\R+\mathfrak{x}_\eps)^{-1}\right)\right)(x+\mathfrak{x}_\eps(x))
\end{align}
where $G^{[\mathfrak{w}_\eps]}_{\ell,\xi}[1+\mathfrak{h}_\eps]$ is defined by  
\[
G^{[\mathfrak{w}_\eps]}_{\ell,\xi}[1+\mathfrak{h}_\eps](\Psi)=\d_z\Phi(\cdot,1+\mathfrak{h}_\eps)
\]
with $\Phi$ the bounded solution to
\begin{align*}
((\d_x+\iD\xi)^2+\d_z^2+(\iD\ell)^2\,(1+\mathfrak{w}_\eps))\Phi&=0\ \textrm{ in }\R\times (0,1+\mathfrak{h}_\eps)\,,&
\d_z\Phi(\cdot,0)&=0\,,&
\Phi(\cdot,1+\mathfrak{h}_\eps)&=\Psi\,.
\end{align*}
Sharp regularity properties of $G^{[\mathfrak{w}_\eps]}_{\ell,\xi}[1+\mathfrak{h}_\eps]$ are straightforward to prove ; possibly nonsharp estimates arise when estimating the effect of compositions without tracking cancellations. The presence of exponential factors in \eqref{conformal} is due to the fact that Floquet parameters in original physical variables and conformal variables are different whereas the presence of $\mathfrak{v}_\eps$ and $\mathfrak{w}_\eps$ arise from the fact that conformal variables preserve angles but not lengths.

\subsubsection*{Spectral decomposition and asymptotics}

Note that unlike $L^\eps$, Bloch symbols $L^\eps_{\ell,\xi}$ are elliptic (of first order), hence the easy identification of their domain. On a related note, the main advantage when replacing the direct analysis of $L^\eps$ with those of $L^\eps_{\ell,\xi}$ is that each  $L^\eps_{\ell,\xi}$ has compact resolvents thus spectrum reduced to eigenvalues of finite multiplicity, arranged discretely. A key related expectation is that
\begin{align}\label{spec}
\sigma(L^\eps)=\bigcup_{(\ell,\xi)\in\R\times[-\kappa/2,\kappa/2)}\sigma_{{\rm per}}(L^\eps_{\ell,\xi}),
\end{align}
where we have added the suffix ${\rm per}$ to mark that each $L_{\ell,\xi}^\eps$ acts on functions defined over $\R/T\Z$. The inclusion $\supset$ of equality \eqref{spec} follows from a systematic abstract argument as in \cite[p.30-31]{R} or in \cite[Appendix~A.2]{R2D2}. Indeed if $\lambda_0\in\sigma_{{\rm per}}(L^\eps_{\ell_0,\xi_0})$, then one may pick a corresponding eigenfunction $U_0$, and define $U^{[\delta]}$ by
\[
U^{[\delta]}(x,y)\ :=\ \int_{\xi_0-\delta}^{\xi_0+\delta}\int_{\ell_0-\delta}^{\ell_0+\delta} \eD^{\iD\xi x+\iD\ell y}\ U_0(x)\ \dif\ell\,\dif\xi\,,
\]
so that
\begin{align*}
\frac{\|(\lambda_0\,I-L^\eps)\,U^{[\delta]}\|_{L^2\times H^{\frac12}(\R^2)}}{\|U^{[\delta]}\|_{L^2\times H^{\frac12}(\R^2)}}
&\approx \frac{\left\|(\xi,\ell,x)\mapsto {\bf1}_{B((\xi_0,\ell_0),\delta)}(\xi,\ell) (L^\eps_{\xi_0,\ell_0} - L^\eps_{\xi,\ell}) (U_{0})(x) \right\|_{L_{(\xi,\ell)}^2(L^2_{{\rm per},\kappa}\times H^{\frac12}_{{\rm per},\kappa})}}{\left\|(\xi,\ell,x)\mapsto {\bf1}_{B((\xi_0,\ell_0),\delta)}(\xi,\ell)\ U_{0}(x) \right\|_{L_{(\xi,\ell)}^2(L^2_{{\rm per},\kappa}\times H^{\frac12}_{{\rm per},\kappa})}}\\
&\stackrel{\delta\to0}{\longrightarrow} 0\,,
\end{align*}
hence $\lambda_0\in\sigma(L^\eps)$. Note that, consistently with the above discussion, here we use only the continuity of $(\ell,\xi)\mapsto L^\eps_{\ell,\xi}$ when applied to a function $U_0$ with high regularity, being an eigenfunction of the elliptic operator $L^\eps_{\ell_0,\xi_0}$. 
The elucidation of the reverse inclusion would require a study of the high-frequency $|\ell|\to\infty$ behavior of $L^\eps_{\ell,\xi}$ but is fortunately not needed to prove Theorem~\ref{th:main}.

\section{The zero-amplitude spectrum}

\subsection{A curve of double eigenvalues}

We now begin our instability analysis by identifying a suitable curve of double eigenvalues of 
\[
L^0_{\ell,\xi}= \bp
   \d_x+\iD\xi  &|D|_{\ell,\xi}\tah(|D|_{\ell,\xi}) \\
   -\mu_0 &\d_x+\iD\xi \ep\,.
\]
By using Fourier series as we have used the Fourier/Floquet-Bloch transform in the previous section one obtains
\begin{align}\label{spec0}
\sigma_{{\rm per}}(L^0_{\ell,\xi})=\bigcup_{n\in\Z}\sigma\left(
\bp
   \iD\,n\kappa+\iD\xi  &|n\kappa|_{\ell,\xi}\tah(|n\kappa|_{\ell,\xi}) \\
   -\mu_0 &\iD\,n\kappa+\iD\xi \ep
\right)=\bigcup_{n\in\Z}\{\,\lambda_{n,\ell,\xi}^-\,,\,\lambda_{n,\ell,\xi}^+\,\}
\end{align}
with 
\be\label{def:lambdan}
\lambda_{n,\ell,\xi}^\pm :=\iD\left(
n\kappa+\xi\pm \sqrt{\mu_0\,|n\kappa|_{\ell,\xi}\tah(|n\kappa|_{\ell,\xi})}
\right)\,.
\ee
Note that unlike for the prospective \eqref{spec}, the high-frequency $|n|\to\infty$ analysis required to prove the inclusion $\subset$ of \eqref{spec0} is straightforward.

To motivate the main result of the present section we point out that for any $(n,\ell,\xi)$
\begin{align*}
\lambda_{n,-\ell,\xi}^\pm&=\lambda_{n,\ell,\xi}^\pm\,,&
\lambda_{-n,\ell,-\xi}^\pm&=-\lambda_{n,\ell,\xi}^\mp\,,
\end{align*}
and that the choice of $\mu_0$ from Proposition~\ref{p:Stokes} is precisely done to enforce
\[
\lambda^-_{0,0,0}\,=\,\lambda^+_{0,0,0}\,=\,\lambda^-_{1,0,0}\,=\,\lambda^+_{-1,0,0}\,=\,0\,.
\]
The goal of this section is the following proposition.

\begin{proposition}\label{p:double}
There exist $\ell_0>0$, $C_0>0$ and $\xi_+:(-\ell_0,\ell_0)\to (-\kappa/2,\kappa/2)$ analytic, odd, with a simple zero at $0$ and such that for any $\ell\in(-\ell_0,\ell_0)$,
\begin{align*}
\lambda^-_{1,\ell,\xi_+(\ell)}=\lambda^+_{-1,\ell,\xi_+(\ell)}&=:\iD\sigma_+(\ell)\,,&
|\lambda^\pm_{0,\ell,\xi_+(\ell)}-\iD\sigma_+(\ell)|&\geq\,C_0\,|\ell|
\end{align*}
and when $(n,\#)\notin\{(0,+),(0,-),(1,-),(-1,+)\}$
\[
|\lambda^\#_{n,\ell,\xi_+(\ell)}-\iD\sigma_+(\ell)|\,\geq\,C_0\,.
\]
\end{proposition}
Symmetrically, Proposition \ref{p:double} holds with $\xi_+(\ell)$ and $\sigma_+(\ell)$ replaced with $\xi_-(\ell):=-\xi_+(\ell)$ and $\sigma_-(\ell):=-\sigma_+(\ell)$, respectively. 
For $\kappa=1$, we visualize $\xi_\pm(\ell)$ and $\sigma_\pm(\ell)$ in Figure~\ref{figure1}.
\begin{figure}[htbp]
\centering
\includegraphics[scale=0.4]{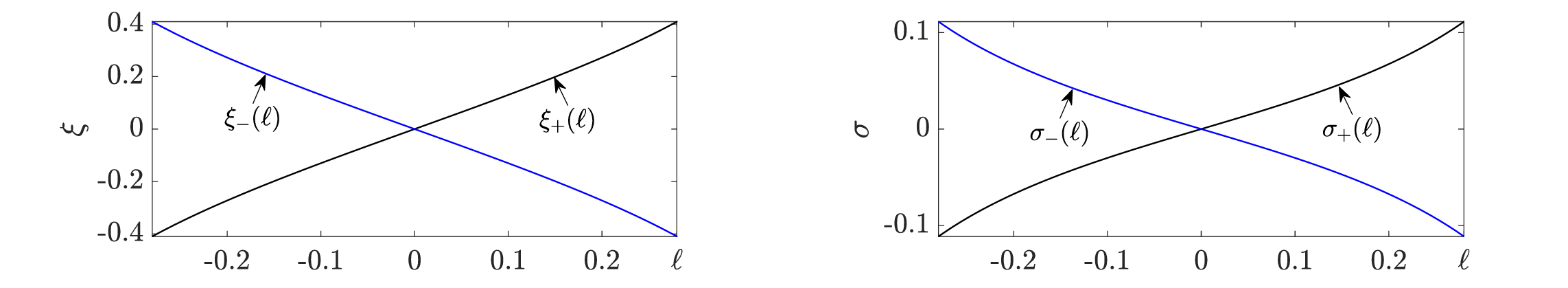}
\caption{The eigenvalues $\lambda_{1,\ell,\xi}^-$ and $\lambda_{-1,\ell,\xi}^+$ of $L_{\ell,\xi}^0$ collide on the curves $(\ell,\xi_+(\ell))$ and $(\ell,\xi_-(\ell))$ where they are equal to $i\sigma_+(\ell)$ and $i\sigma_-(\ell)$, respectively. Left panel: Visualization of  $(\ell,\xi_\pm(\ell))$ for $\kappa=1$. Right panel: Visualization of  $(\ell,\sigma_\pm(\ell))$ for $\kappa=1$.}
\label{figure1}
\end{figure}

From now on, we will only consider spectra near $\iD\sigma_+(\ell)$. 
Proposition~\ref{p:double} is obtained by combining a series of lemmas. As a preliminary we study an auxiliary function.

\begin{lemma}\label{F-aux} The map $\R_+\to\R$, $r\mapsto\sqrt{r\tah(r)}$ is concave and possesses an extension $F:\R\to\R$, that is analytic, increasing, odd, and such that $F'$ does not vanish and $F''$ vanishes only at the origin where it has a simple zero.
\end{lemma}
\begin{proof}
Most properties stem readily from the fact that $\tah$ is analytic, increasing, odd, with a simple zero at the origin and when $r>0$, $\sqrt{r\tah(r)}=r\sqrt{\tah(r)/r}$. We only need to justify that when $r>0$, $F''(r)<0$. Direct computations show that for any $r>0$
\begin{align*}
F''(r)&=\frac{\tF(r)}{4r^{3/2}(\eD^{2r} - 1)^{3/2}(\eD^{2r} + 1)^{5/2}},
\end{align*}
where $\tF:\R\to\R$ is defined by
\begin{align*}
\tF(r)&:=-\eD^{8r}-16r^2\eD^{6r}+8r\eD^{6r}+16r^2e^{4r}+2\eD^{4r}-16r^2\eD^{2r}-8r\eD^{2r}- 1\,.
\end{align*}
Direct computation reveals $\tilde{F}(0),\tilde{F}'(0),\tilde{F}''(0),\tilde{F}'''(0)=0$ and for any $r>0$
\begin{align*}
\tF''''(r)&=-128\eD^{2r}\left(\eD^{6r}(32-28\eD^{-4r})+r^2\eD^{4r}(162-32\eD^{-2r})+r\eD^{4r}(135-64e^{-2r})+2r^2+9r+8\right)\\
&<\,0.
\end{align*}
From the integral Taylor formula it follows that $\tF$ is negative on $\R_+^*$ thus so is $F''$.
\end{proof}

We now prove the separation part of Proposition~\ref{p:double}.

\begin{lemma}
There exist $\delta_1>0$ and $C_1>0$ such that for any $(\ell,\xi)\in[-\delta_1,\delta_1]\times[-\delta_1,\delta_1]$,
\begin{align*}
|\lambda^\pm_{0,\ell,\xi}-\lambda^\pm_{\mp1,\ell,\xi}|&\geq\,C_1\,|(\ell,\xi)|
\end{align*}
and when $(n,\#)\notin\{(0,+),(0,-),(1,-),(-1,+)\}$
\[
|\lambda^\#_{n,\ell,\xi}|\,\geq\,C_1\,.
\]
\end{lemma}

\begin{proof}
Locally uniformly in $(\ell,\xi)$, $\lambda^\pm_{n,\ell,\xi}\stackrel{|n|\to\infty}{\sim}\iD n\kappa$. Thus in order to prove the second bound it is sufficient to invoke continuity in $(\ell,\xi)$ and check that when $(n,\#)\notin\{(0,+),(0,-),(1,-),(-1,+)\}$, $\lambda^\#_{n,0,0}\neq 0$. This is obvious when $\#n>0$ whereas when $\#n<0$ it follows from the fact that, with $F$ as in Lemma~\ref{F-aux}, $F_{\mu_0}:r\mapsto -r+\sqrt{\mu_0}\,F(r)$ is strictly concave on $\R_+$ and thus cannot vanish more than twice (here at $0$ and $\kappa$). 

Concerning the first bound, we first deduce from triangle inequality that $|\pm\kappa|_{\ell,\xi}\leq \kappa+|(\ell,\xi)|$ so that 
\begin{align*}
\mp\iD\,(\lambda^\pm_{0,\ell,\xi}-\lambda^\pm_{\mp1,\ell,\xi})
&=\kappa-\sqrt{\mu_0}\,\big(F(|\mp\kappa|_{\ell,\xi})+F(|(\ell,\xi)|)\big)\\
&\geq \kappa-\sqrt{\mu_0}\,\big(F(\kappa+|(\ell,\xi)|)+F(|(\ell,\xi)|)\big)
=-F_{\mu_0}(\kappa+|(\ell,\xi)|)+F_{\mu_0}(|(\ell,\xi)|)\,.
\end{align*}
Then we derive the first bound from the strict concavity of $F_{\mu_0}$ on $\R_+$ which implies that $F_{\mu_0}$ is positive on $(0,\kappa)$ and has negative derivative at its zero $\kappa$. 
\end{proof}

We conclude the proof of Proposition~\ref{p:double} by proving the existence of a curve of double eigenvalues.

\begin{lemma}\label{l:double}
There exist $\ell_0>0$ and $\xi_+:(-\ell_0,\ell_0)\to (-\kappa/2,\kappa/2)$ analytic, odd, with a simple zero at $0$ and such that for any $\ell\in(-\ell_0,\ell_0)$,
\begin{align*}
\lambda^-_{1,\ell,\xi_+(\ell)}=\lambda^+_{-1,\ell,\xi_+(\ell)}\,.
\end{align*}
\end{lemma}

\begin{proof}
Note that with $F$ as in Lemma~\ref{F-aux}
\[
-\iD\,(\lambda^-_{1,\ell,\xi}-\lambda^+_{-1,\ell,\xi})
=2\kappa-\sqrt{\mu_0}\,(F(|(\ell,\kappa-\xi)|)+F(|(\ell,\kappa+\xi)|))=:\varphi(\ell,\xi)
\]
defines in a neighborhood of $(0,0)$ an analytic function of $(\ell,\xi)$ that is even with respect to both $\ell$ and $\xi$, thus it also defines an analytic function of $(\ell^2,\xi^2)$. We compute 
\begin{align*}
\d_\ell\varphi(\ell,0)&=
-2\sqrt{\mu_0}\,F'(|(\ell,\kappa)|)\frac{\ell}{|(\ell,\kappa)|}\,,&
\d_\ell^2\varphi(0,0)
&=\,-2\sqrt{\mu_0}\,\frac{F'(\kappa)}{\kappa}<0
\end{align*} 
\begin{align*}
\d_\xi\varphi(0,\xi)&=
-\sqrt{\mu_0}\,\left(-F'(\kappa-\xi)+F'(\kappa+\xi)\right)\,,&
\d_\xi^2\varphi(0,0)
&=\,-2\sqrt{\mu_0}\,F''(\kappa)>0
\end{align*} 
and set
\[
\underline{\Xi}:=-\frac{\d_\ell^2\varphi(0,0)}{\d_\xi^2\varphi(0,0)}>0.
\]
As a consequence there exists $\Phi$ an analytic function defined on $\Omega$ a neighborhood of $(0,\underline{\Xi})$ such that 
\begin{align*}
\Phi(0,\underline{\Xi})&=0\,,& 
\d_\Xi\Phi(0,\underline{\Xi})&=\frac12\d_\xi^2\varphi(0,0)\neq0\,
\end{align*}
and when $(\ell,\xi,\Xi)$ is such that $\xi^2=\Xi\,\ell^2$ and $(\ell^2,\Xi)\in \Omega$
\[
\varphi(\ell,\xi)\,=\,\ell^2\,\Phi(\ell^2,\Xi)\,.
\]
By applying the implicit function theorem to $\Phi$ one obtains $\Xi_+$ analytic on a neighborhood of $0$ such that $\Xi_+(0)=\underline{\Xi}$ and for any small $\Lambda$, $\Phi(\Lambda,\Xi_+(\Lambda))=0$. The proof is concluded by setting $\xi_+(\ell):=\ell\,\sqrt{\Xi_+(\ell^2)}$. 
\end{proof}

\subsection{A diagonalization basis at zero-amplitude}\label{ss:choice}

From now on we always assume implicitly that $\ell$ is sufficiently small to be in the application range of Proposition~\ref{p:double}.

In the end we shall analyze how the spectrum of $L^\eps_{\ell,\xi}$ behaves in a ball $B(\iD\sigma_+(\ell),c_0\,|\ell|)$, centered at $\iD\sigma_+(\ell)$ with radius $c_0 |\ell|$, when $(\eps,\xi)$ is $\cO(|\ell|)$-close to $(0,\xi_+(\ell))$, with $c_0>0$ sufficiently small to guarantee that only two eigenvalues (counted with algebraic multiplicity) lie in the ball. To do so, we need to pick a basis for the associated generalized eigenspace. We present here the choice we make at zero amplitude.

We introduce as eigenvectors of $L^0_{\ell,\xi}$ respectively for $\lambda_{1,\ell,\xi}^{-}$ and $\lambda_{-1,\ell,\xi}^{+}$
\begin{align} \label{basisofL0}
q_-^0(\ell,\xi,\cdot)&:=\eD^{i\kappa\,\cdot\,} \bp \iD\alpha_-(\ell, \xi)\\ 1\ep\,,&
q_+^0(\ell, \xi,\cdot)&:=\eD^{-\iD\kappa\,\cdot\,} \bp-\iD\alpha_+(\ell, \xi)\\1\ep\,, 
\end{align}
with
\be\label{def-omega-alphapm}
\alpha_{\pm}(\ell, \xi):=\frac{\sqrt{|\mp \kappa|_{\ell,\xi} \,\tah(|\mp \kappa|_{\ell,\xi})}}{\sqrt{\mu_0}}
=\frac{F(|\mp \kappa|_{\ell,\xi})\,F(\kappa)}{\kappa}\,,
\ee 
where $F$ is as in Lemma~\ref{F-aux}.

Before providing a dual basis at zero-amplitude, we observe that the Bloch symbol $L^\eps_{\ell,\xi}$ also inherits the Hamiltonian symmetry 
\begin{align*}
L^\eps_{\ell,\xi}&=J\,A^\eps_{\ell,\xi}\,,&
A^\eps_{\ell,\xi}&:=\bp
    -(\mu_0+\underline{a}_\eps)&(1-\underline{V}_\eps)\,(\d_x+\iD\xi)\\
    -(\d_x+\iD\xi)\left((1-\underline{V}_\eps)\,\cdot\,\right)& -G_{\ell,\xi}[1+\underline{\eta}_\eps]
\ep\,,
\end{align*}
with $J$ as in \eqref{Ham} thus skew-symmetric and invertible and $A^\eps_{\ell,\xi}$ symmetric. In particular the generalized eigenspace of the adjoint operator $(L^0_{\ell,\xi})^*$ associated with the spectrum $\{-\lambda_{1,\ell,\xi}^-,-\lambda_{-1,\ell,\xi}^+\}$ is spanned by $J^{-1}q_-^0(\ell,\xi,\cdot)$ and $J^{-1}q_+^0(\ell,\xi,\cdot)$, that are eigenvectors of respective eigenvalues $-\lambda_{1,\ell,\xi}^-$ and $-\lambda_{-1,\ell,\xi}^+$. To derive a basis dual to $(q_-^0(\ell,\xi,\cdot),q_+^0(\ell,\xi,\cdot))$ we deduce normalizing factors from
\begin{align*}
\langle J^{-1}\,q_\pm^0(\ell,\xi, \cdot),q_\mp^0(\ell,\xi,\cdot)\rangle&=0\,,&
\langle \,\iD J^{-1}q_\pm^0(\ell,\xi,\cdot), q_\pm^0(\ell,\xi,\cdot)\rangle&=\mp2\alpha_\pm(\ell, \xi)\,.
\end{align*}
In the foregoing and henceforth we use $\langle\,\cdot\,,\,\cdot\,\rangle$ to denote canonical complex inner products, with the convention that the product is skew-linear in its first component and linear in its second. On $L^2(\R/T\Z;\C^2)$, the convention is
\[
\langle f, g \rangle:= \frac{1}{T}\int_0^T \langle f(x),\,g(x)\rangle\,\dif x
\,=\,\frac{1}{T}\int_0^T\overline{f(x)}\cdot g(x)\,\dif x\,.
\]

Therefore we obtain the sought dual basis $(\tq_-^0(\ell,\xi,\cdot),\tq_+^0(\ell,\xi,\cdot))$ by setting
\begin{align*}
\tq_-^0(\ell,\xi,\cdot)&:=\frac{\iD J^{-1}q_-^0(\ell,\xi,\cdot)}{2\alpha_-(\ell, \xi)}\,,&
\tq_+^0(\ell,\xi,\cdot)&:=-\frac{\iD J^{-1}q_+^0(\ell,\xi,\cdot)}{2\alpha_+(\ell, \xi)}\,.
\end{align*}

\section{Instability index}

\subsection{Definition of the index}

We now extend to small $\eps$ the bases of Subsection~\ref{ss:choice} with Kato's transformation functions \cite[Chapter~Two, \S 4]{Katobook}. Explicitly, from Proposition~\ref{p:double} and the continuity of $L^\eps_{\ell,\xi}$ with respect to $(\ell,\xi)$ there exist positive $\ell_0$, $c_0$, $c_0'$ such that when $0<|\ell|<\ell_0$ and $|\big(\eps,\xi-\xi_+(\ell)\big)|\leq c_0'\,|\ell|$, the Riesz spectral projector 
\[ 
\Pi_{\ell,\xi}(\eps):=\frac{1}{2\pi \iD} \oint_{\d B(\,\iD\sigma_+(\ell),c_0\,|\ell|)}
(\lambda I-L_{\ell,\xi}^\eps)^{-1}\,\dif\lambda
\]
is well defined and thus one obtains\footnote{Extension operators are non unique and various processes are proposed in \cite{Katobook}.} an extension operator $\cU_{\ell,\xi}(\eps)$ by solving the Cauchy problem
\begin{align*}
\d_{\eps}\cU_{\ell,\xi}\,
&=\big[\d_\eps\Pi_{\ell,\xi},\Pi_{\ell,\xi}\big]\,\cU_{\ell,\xi}\,,& 
\cU_{\ell,\xi}(0)&=I\,,
\end{align*}
with $[\cdot,\cdot]$ denoting commutator. Then we set
\begin{align*}
q_-^\eps(\ell,\xi,\cdot)&:=\cU_{\ell,\xi}(\eps)\,q_-^0(\ell,\xi,\cdot)\,,& 
q_+^\eps(\ell,\xi,\cdot)&:=\cU_{\ell,\xi}(\eps)\,q_+^0(\ell,\xi,\cdot)\,,
\end{align*}
so that $(q_-^\eps(\ell,\xi,\cdot),q_+^\eps(\ell,\xi,\cdot))$ forms an ordered basis of the range of $\Pi_{\ell,\xi}(\eps)$, the generalized eigenspace of $L^\eps_{\ell,\xi}$ associated with its spectrum in $B(\,\iD\sigma_+(\ell),c_0\,|\ell|)$. 

Likewise one may solve a Cauchy problem for the projector $\Pi_{\ell,\xi}(\eps)^*$ to obtain a dual extension operator $\cV_{\ell,\xi}(\eps)$, that turns out to be equal to $(\cU_{\ell,\xi}(\eps)^*)^{-1}$. Then we obtain a suitable dual basis by setting
\begin{align*}
\tq_-^\eps(\ell,\xi,\cdot)&:=\cV_{\ell,\xi}(\eps)\,\tq_-^0(\ell,\xi,\cdot)\,,& 
\tq_+^\eps(\ell,\xi,\cdot)&:=\cV_{\ell,\xi}(\eps)\,\tq_+^0(\ell,\xi,\cdot)\,.
\end{align*}
From the Hamiltonian symmetry one then deduces that one still has
\begin{align*}
\tq_-^\eps(\ell,\xi,\cdot)&=\frac{\iD J^{-1}q_-^\eps(\ell,\xi,\cdot)}{2\alpha_-(\ell, \xi)}\,,&
\tq_+^\eps(\ell,\xi,\cdot)&=-\frac{\iD J^{-1}q_+^\eps(\ell,\xi,\cdot)}{2\alpha_+(\ell, \xi)}\,.
\end{align*}
We refer to \cite{NRS-EP} for details on the latter. Alternatively to skip the dual extension we could have proved that normalization relations are independent of $\eps$ and defined $(\tq_-^\eps(\ell,\xi,\cdot),\tq_+^\eps(\ell,\xi,\cdot))$ in this way.

As a result when $0<|\ell|<\ell_0$ and $|\big(\eps,\xi-\xi_+(\ell)\big)|\leq c_0'\,|\ell|$, the matrix
\begin{align*}
 D_{\ell,\xi}^{\eps}= \bp
\langle \tq_+^\eps(\ell, \xi, \cdot), L_{\ell,\xi}^\eps q_+^\eps(\ell, \xi, \cdot)\rangle& 
\langle \tq_+^\eps(\ell, \xi, \cdot), L_{\ell,\xi}^\eps q_-^\eps(\ell, \xi, \cdot)\rangle\\[10pt]
\langle \tq_-^\eps(\ell, \xi, \cdot), L_{\ell,\xi}^\eps q_+^\eps(\ell, \xi, \cdot)\rangle& 
\langle \tq_-^\eps(\ell, \xi, \cdot), L_{\ell,\xi}^\eps q_-^\eps(\ell, \xi, \cdot)\rangle
\ep
\end{align*}
satisfies 
\begin{equation}\label{spec2}
\sigma_{{\rm per}}(L_{\ell,\xi}^\eps)\cap B\big(\iD\sigma_+(\ell),c_0|\ell|\big)
\,=\,\sigma(D_{\ell, \xi}^\eps).
\end{equation}
By design, when $0<|\ell|< \ell_0$ and $|\xi-\xi_+(\ell)|\leq c_0'\,|\ell|$,
\[
D_{\ell,\xi}^0\,=\,
\bp
\lambda_{-1,\ell,\xi}^+&0\\
0&\lambda_{1,\ell, \xi}^-
\ep\,.
\]
From the symmetry of $A_{\ell,\xi}^\eps$ and the relation between direct and dual bases stems that
\[ 
D_{\ell,\xi}^\eps
=\frac{\iD}{2}\bp
-\dfrac{b_+(\ell,\xi,\eps)}{\alpha_+(\ell,\xi)}&
-\dfrac{a(\ell,\xi,\eps)}{\alpha_+(\ell, \xi)}\\[1.5em]
\dfrac{\overline{a(\ell,\xi,\eps)}}{\alpha_-(\ell,\xi)}&
\dfrac{b_-(\ell,\xi,\eps)}{\alpha_-(\ell, \xi)}
\ep
\] 
with
\begin{align} 
a(\ell,\xi,\eps)&:=\langle q_+^\eps(\ell,\xi,\cdot),A_{\ell,\xi}^\eps  q_-^\eps(\ell,\xi,\cdot)\rangle\,,
\label{defa}\\
b_+(\ell,\xi,\eps)&:=\langle q_+^\eps(\ell,\xi,\cdot), A_{\ell,\xi}^\eps q_+^\eps(\ell,\xi,\cdot)\rangle\,,
\nonumber\\
b_-(\ell,\xi,\eps)&:=\langle q_-^\eps(\ell,\xi,\cdot),A_{\ell,\xi}^\eps q_-^\eps(\ell,\xi,\cdot)\rangle\,.
\nonumber
\end{align} 
thus $b_+(\ell,\xi,\eps)\in\R$, $b_-(\ell,\xi,\eps)\in\R$. 

Note that here most mathematical objects blow up in the limit $\ell\to 0$ and we do not try to track how they blow up.

Eigenvalues of $D_{\ell,\xi}^\eps$ are the roots of
\begin{equation}\label{e:char}
\left(\lambda-\frac{\iD}{4}\left(-\frac{b_+}{\alpha_+}
+\frac{b_-}{\alpha_-}\right)\right)^2
\,=\,\frac14\,\left(
-\frac14\left(\frac{b_+}{\alpha_+}
+\frac{b_-}{\alpha_-}\right)^2
+\frac{|a|^2}{\alpha_-\alpha_+}
\right)\,,
\end{equation}
where we have omitted to mark dependencies on $(\ell,\xi,\eps)$. To achieve instability we want to make the right-hand side of the foregoing equation as positive as possible. 

The following lemma makes a first step in this direction.

\begin{lemma}\label{l:curve}
There exist $\ell_*$ such that for any $0<|\ell|\leq \ell_*$, there exist $\eps_\ell>0$ and an analytic map $\xi_\ell:\,(-\eps_\ell,\eps_\ell)\to[\xi_+(\ell)-c_0'|\ell|,\xi_+(\ell)+c_0'|\ell|]$ such that $\xi_\ell(0)=\xi_+(\ell)$ and for any $|\eps|\leq \eps_\ell$
\[
\frac{b_+(\ell,\xi_\ell(\eps),\eps)}{\alpha_+(\ell,\xi_\ell(\eps))}+\frac{b_-(\ell,\xi_\ell(\eps),\eps)}{\alpha_-(\ell,\xi_\ell(\eps))}\,=\,0\,.
\]
\end{lemma}

\begin{proof}
The result follows from the implicit function theorem applied to 
\[
\cZ:\quad(\xi,\eps)\mapsto \frac{b_+(\ell,\xi,\eps)}{\alpha_+(\ell,\xi)}+\frac{b_-(\ell,\xi,\eps)}{\alpha_-(\ell,\xi)}\,.
\]
To check the required assumptions we observe that
\[
\cZ(\xi,0)
\,=\,2\iD\,\left(\lambda_{-1,\ell,\xi}^+-\lambda_{1,\ell, \xi}^-\right)
\,=\,-2\varphi(\ell,\xi)
\]
with $\varphi$ as in the proof of Lemma~\ref{l:double}. Thus $\cZ(\xi_+(\ell),0)=0$ and
\begin{align*}
\d_\xi \cZ(\xi_+(\ell),0)\,=\,-2\,\d_\xi\varphi(\ell,\xi_+(\ell))\stackrel{\ell\to0}{\sim}-2\,\d_\xi^2\varphi(0,0)\,\xi_+(\ell)\neq 0\,.
\end{align*}
Hence the lemma.
\end{proof}

We conclude the present subsection by introducing an $\ell$-dependent instability index.

\begin{definition}
With $\ell_0$ and $\xi_+$ as in Proposition~\ref{p:double} and $a$ as in \eqref{defa}, we define\\
for $0<|\ell|< \ell_0$, 
\[
\Gamma_\ell:=\frac12\,\d_\eps^2 a(\ell,\xi_+(\ell),0)\,.
\]
\end{definition}

\begin{proposition}\label{Prop:Gamma}
There exists $\ell_*>0$ such that if $0<|\ell|\leq \ell_*$ is such that $\Gamma_\ell$ is not zero then there exist positive $\eps_\ell'$ and $c_\ell$ such that for any $0<|\eps|\leq \eps_\ell'$ the Bloch symbol $L^\eps_{\ell,\xi_\ell(\eps)}$ (with $\xi_\ell(\eps)$ as in Lemma~\ref{l:curve}) possesses an eigenvalue of real part larger than $c_\ell\,\eps^2$. 
\end{proposition}

The foregoing proposition reduces the proof of Theorem~\ref{th:main} to the proof that if $\ell$ is sufficiently small and nonzero then $\Gamma_\ell\neq0$, a claim that we shall prove in the next subsection.

\begin{proof}
In view of \eqref{spec2}, \eqref{e:char} and Lemma~\ref{l:curve} it is sufficient to prove that
\[
a(\ell,\xi_\ell(\eps),\eps)
\,\stackrel{\eps\to0}{=}\,\Gamma_\ell\,\eps^2\,+\,\cO(\eps^3)\,.
\]
Since $\xi_\ell(\eps)\stackrel{\eps\to0}{=}\xi_+(\ell)+\cO(\eps)$ it is sufficient to prove
\begin{align*}
\d_\xi a(\ell,\xi_+(\ell),0)&=0\,,&
\d_\xi^2 a(\ell,\xi_+(\ell),0)&=0\,,&
\d_\eps a(\ell,\xi_+(\ell),0)&=0\,,&
\d_{\xi\,\eps}^2 a(\ell,\xi_+(\ell),0)&=0\,.
\end{align*}
First we recall that $D^0_{\ell,\xi}$ is diagonal for any $\xi$, hence $a(\ell,\cdot,0)\equiv0$, which implies the first two vanishings. In order to deduce the remaining vanishings thus the lemma we now prove that $\d_\eps a(\ell,\cdot,0)\equiv0$.

To do so we borrow from \cite{NRS-EP} a piece of notation that we shall use heavily in the next subsection. When $M$ is an operator on $L^2(\R/T\Z)\times L^2(\R/T\Z)$ with domain including trigonometric polynomials and $(m,p)\in\Z^2$, we define an operator $M^m_p:\C^2\to\C^2$ by
\be\label{def-Mmp}
M^m_p\,(X):=\langle \eD^{\iD\,(m+p)\,\kappa\,\cdot},M(\eD^{\iD\,m\,\kappa\,\cdot}\,X)\rangle\,.
\ee
Then since, with notation from Proposition~\ref{p:Stokes}, $(\eta_1,\psi_1)\in \cos(\kappa\,\cdot)\,\C^2$ we observe that 
\begin{align*}
\left(\d_\eps (L_{\ell,\xi}^\eps)|_{\eps=0}\right)^p_r&\equiv 0&\textrm{thus}&&
\left(\d_\eps\Pi_{\ell,\xi}(0)\right)^p_r&\equiv 0&
\textrm{when }|r|>1\,.
\end{align*}
Through
\begin{align*}
q_\pm^\eps(\ell, \xi,\cdot)&\,=\,\Pi_{\ell,\xi}(\eps)(q_\pm^
\eps(\ell, \xi,\cdot))\,,&
L_{\ell,\xi}^\eps q_\pm^\eps(\ell, \xi,\cdot)&\,=\,\Pi_{\ell,\xi}(\eps)(L_{\ell,\xi}^\eps q_\pm^\eps(\ell, \xi,\cdot))\,,&
\end{align*}
this implies that
\begin{align*}
\d_\eps q_\pm^0(\ell, \xi,\cdot)\,,\
\d_\eps(\eps\mapsto L_{\ell,\xi}^\eps q_\pm^\eps(\ell, \xi,\cdot))_{|\eps=0}
&\in \Span\left(\left\{\,\eD^{\iD\,(\mp1+p)\,\kappa\,\cdot\,}\,X\,;\,|p|\leq 1\,,\,
X\in\C^2\,\right\}\right)\,.
\end{align*}
By orthogonality of trigonometric monomials this proves the claim hence the proposition.
\end{proof}

\subsection{Computation of the index}

Our goal is now to gain information on $\Gamma_\ell$.

As a preliminary we complete our notational set. First we denote $V_\pm(\ell)$ the vectors of $\C^2$ appearing in the definitions of $q_\pm^0(\ell, \xi_+(\ell),\cdot)$, explicitly
\begin{align*}
V_\pm(\ell)&:=\begin{pmatrix}\mp\iD\alpha_{\pm}(\ell,\xi_+(\ell))\\ 1\end{pmatrix}\,.
\end{align*}
Then we introduce matrices $\cM_{\pm}(\ell)$ encoding $\Pi_{\ell,\xi_+(\ell)}^0$,
\begin{align*}\label{def-cMpm}
\cM_{\pm}(\ell)
&:=\frac{1}{2}\bp
1 & \mp \iD\,\alpha_\pm(\ell,\xi_+(\ell)) \\[5pt]
\mp\frac{1}{\iD\,\alpha_\pm(\ell,\xi_+(\ell))} & 1
\ep\,,
\end{align*}
so that when $(m,X)\in\Z\times \C^2$
\begin{align*}
\Pi^0_{\ell,\xi_+(\ell)} (\eD^{\iD\kappa m \,\cdot}X)
\,=\, \begin{cases}
\,\eD^{\iD\kappa\,\cdot}\cM_-(\ell)\,X\,,&\textrm{if }m=1\,,\\
\,\eD^{-\iD\kappa\,\cdot}\cM_+(\ell)\,X\,,&\textrm{if }m=-1\,,\\
\,0\,, &\textrm{otherwise}\,.
\end{cases}
\end{align*}
Likewise we introduce the matrix
\[
A_0(\ell):=\left(\iD\,\sigma_+(\ell)\,I-
\bp
\iD\xi_+(\ell)  &|(\ell,\xi_+(\ell))|\tah(|(\ell,\xi_+(\ell))|) \\
-\mu_0 &\iD\xi_+(\ell) \ep
\right)^{-1}
\]
to encode the action on constant functions of the non singular part of $(\lambda I-L^0_{\ell,\xi_+(\ell)})^{-1}$ at $\lambda=\iD\,\sigma_+(\ell)$. Finally for $m\in\{0,1,2\}$, we set
\begin{align*}
L_\ell^{[m]}&:=\frac{1}{m!}\d_\eps^m (L_{\ell,\xi_+(\ell)}^\eps)|_{\eps=0}\,,&
\cU^{[m]}_\ell&:=\frac{1}{m!}\d_\eps^m \cU_{\ell,\xi_+(\ell)}(0)\,,&
\Pi^{[m]}_\ell&:=\frac{1}{m!}\d_\eps^m \Pi_{\ell,\xi_+(\ell)}(0)\,.
\end{align*}

The following lemma keeps track of the Hamiltonian symmetry.

\begin{lemma}\label{l:dual}
The operator 
\[
S:\ \cM_{2\times 2}(\C)\mapsto \cM_{2\times 2}(\C)\,,\qquad
B\mapsto J\,B^*\,J^{-1}
\]
satisfies 
\begin{align*}
S[\cM_{\pm}(\ell)]&=\cM_{\pm}(\ell)\,,&
S[A_0(\ell)]&=-A_0(\ell)\,,
\end{align*}
and for any $(p,q)\in\Z$, $m\in\{0,1,2\}$,
\begin{align*}
S[(L_\ell^{[m]})^p_q]&=-(L_\ell^{[m]})^{p+q}_{-q}\,,&
S[(\Pi_\ell^{[m]})^p_q]=((\Pi_\ell^{[m]})^{p+q}_{-q}\,.
\end{align*}
\end{lemma}

\begin{proof}
This follows readily from the obvious $(M^m_p)^*=(M^*)^{m+p}_{-p}$ and the Hamiltonian symmetry $J\,(L_{\ell,\xi_+(\ell)}^\eps)^*\,J^{-1}=-L_{\ell,\xi_+(\ell)}^\eps$.
\end{proof}

Our main step in the evaluation of $\Gamma_\ell$ is the following proposition, analogous to\footnote{We recall that when quoting \cite{NRS-EP} we refer to the longer arXiv preprint {\tt arXiv:2210.10118} rather than to the shortened published version.} \cite[Proposition~5.5]{NRS-EP}.

\begin{proposition}\label{p:Nformula}
For $0<|\ell|< \ell_0$,
\ba\label{Gamma-exp}
\Gamma_\ell&=-\langle J^{-1}\,V_+(\ell), \cN_{\ell} V_-(\ell)\rangle\,,&
\text{ with}&&
\cN_{\ell}&:=(L_\ell^{[2]})^{1}_{-2}+ (L_\ell^{[1]})_{-1}^0 A_0(\ell)(L_\ell^{[1]})_{-1}^1\,.
\ea
\end{proposition}

\begin{proof}
A direct expansion gives
\begin{align*}
\Gamma_{\ell}\,=&-\big\langle J^{-1}(\cU_\ell^{[2]})_{2}^{-1}V_+(\ell), (L_\ell^{[0]})_0^1 V_-(\ell)\big\rangle
-\big\langle J^{-1} (\cU_\ell^{[1]})_{1}^{-1}V_+(\ell),  \big((L_\ell^{[1]})_{-1}^{1}+(L_\ell^{[0]})_0^{0}\, (\cU_\ell^{[1]})_{-1}^{1}\big) V_-(\ell)\big\rangle\\
&-\big\langle J^{-1} V_+(\ell),  \big((L_\ell^{[2]})_{-2}^{1} +(L_\ell^{[1]})_{-1}^{0}(\cU_\ell^{[1]})_{-1}^{1}+(L_\ell^{[0]})_0^{-1} (\cU_\ell^{[2]})_{-2}^{1} \big)V_-(\ell)\big\rangle.
\end{align*}
Now, expanding the Cauchy problem defining $\cU_{\ell,\xi}$ yields
\begin{align*}
\cU_\ell^{[1]}&=[\Pi_\ell^{[1]},\Pi_\ell^{[0]}]\,,&
\cU_\ell^{[2]}&=
\big[\Pi_\ell^{[2]}, \Pi_\ell^{[0]}\big]
+\frac12\,[\Pi_\ell^{[1]}, \Pi_\ell^{[0]}\big]^2\,.
\end{align*}

To compute $\Pi_\ell^{[1]}$, we use the identity
\[
\d_\eps\left(\eps\mapsto\big(\lambda I-L_{\ell,\xi_+(\ell)}^{\eps}\big)^{-1}\right)\big|_{\eps=0}
\,=\,(\lambda\,I-L_\ell^{[0]})^{-1}\,L_\ell^{[1]}\,(\lambda\,I-L_\ell^{[0]})^{-1}
\]
and the Cauchy residue theorem. In this way one obtains
\begin{align*}
(\Pi_\ell^{[1]})_1^{-1}&=A_0(\ell)\,(L_\ell^{[1]})_1^{-1}\,\cM_+(\ell)\,,& 
(\Pi_\ell^{[1]})_{-1}^{1}&=A_0(\ell)\,(L_\ell^{[1]})^1_{-1}\,\cM_-(\ell)\,,\\
(\Pi_\ell^{[1]})^0_{-1}&=\cM_+(\ell)\,(L_\ell^{[1]})^0_{-1}\,A_0(\ell)\,,&
(\Pi_\ell^{[1]})_{1}^{0}&= \cM_-(\ell)\,(L_\ell^{[1]})_1^{0}\,A_0(\ell)\,.
\end{align*}
This also implies that
\begin{align*}
(\cU_\ell^{[1]})_{-1}^1&=A_0(\ell)\,(L_\ell^{[1]})_{-1}^{1}\,\cM_-(\ell)\,,&
(\cU_\ell^{[1]})_{1}^{-1}&=A_0(\ell)\,(L_\ell^{[1]})_{1}^{-1}\,\cM_+(\ell)\,,
\end{align*}
and 
\begin{align*}
\big(\cU_\ell^{[2]}-\big[\Pi_\ell^{[2]},\Pi_\ell^{[0]}\big]\big)_{-2}^1
&=\frac12[\Pi_\ell^{[1]}, \Pi_\ell^{[0]}\big]_{-1}^0\,[\Pi_\ell^{[1]}, \Pi_\ell^{[0]}\big]_{-1}^1
=-\frac12\cM_+(\ell)\,(L_\ell^{[1]})_{-1}^{0}\,A_0(\ell)^2\,(L_\ell^{[1]})_{-1}^{1}\,\cM_-(\ell)\,,\\
\big(\cU_\ell^{[2]}-\big[\Pi_\ell^{[2]},\Pi_\ell^{[0]}\big]\big)_{2}^{-1}
&=\frac12[\Pi_\ell^{[1]}, \Pi_\ell^{[0]}\big]_{1}^0\,[\Pi_\ell^{[1]}, \Pi_\ell^{[0]}\big]_{1}^{-1}
=-\frac12\cM_-(\ell)\,(L_\ell^{[1]})_{1}^{0}\,A_0(\ell)^2\,(L_\ell^{[1]})_{1}^{-1}\,\cM_+(\ell)\,\,.
\end{align*}

Going back to $\Gamma_\ell$, we note that as a result
\[
\big\langle J^{-1} V_+(\ell),\big((L_\ell^{[2]})_{-2}^{1} +(L_\ell^{[1]})_{-1}^{0}(\cU_\ell^{[1]})_{-1}^{1} \big)V_-(\ell)\big\rangle
=\big\langle J^{-1} V_+(\ell),\cN_\ell\,V_-(\ell)\big\rangle\,.
\]
Then Lemma~\ref{l:dual} implies that
\[
\big\langle J^{-1}\big(\big[\Pi_\ell^{[2]},\Pi_\ell^{[0]}\big]\big)_{2}^{-1}V_+(\ell), (L_\ell^{[0]})_0^1 V_-(\ell)\big\rangle
+\big\langle J^{-1} V_+(\ell),(L_\ell^{[0]})_0^{-1} \big(\big[\Pi_\ell^{[2]},\Pi_\ell^{[0]}\big]\big)_{-2}^{1}V_-(\ell)\big\rangle
\,=\,0
\]
so that
\begin{align*}
\big\langle J^{-1}(\cU_\ell^{[2]})_{2}^{-1}V_+(\ell)&, (L_\ell^{[0]})_0^1 V_-(\ell)\big\rangle
+\big\langle J^{-1} V_+(\ell),(L_\ell^{[0]})_0^{-1} (\cU_\ell^{[2]})_{-2}^{1}V_-(\ell)\big\rangle\\
&\,=\,
-\iD\sigma_+(\ell)\,\big\langle J^{-1} V_+(\ell),
\cM_+(\ell) (L_\ell^{[1]})_{-1}^{0}\,A_0(\ell)^2\,(L_\ell^{[1]})_{-1}^{1}\,\cM_-(\ell) V_-(\ell)\big\rangle\,.
\end{align*}
Finally, by using Lemma~\ref{l:dual} again, one derives that
\begin{align*}
\big\langle J^{-1}(\cU_\ell^{[1]})_{1}^{-1}V_+(\ell)&,\big((L_\ell^{[1]})_{-1}^{1}+(L_\ell^{[0]})_0^{0} (\cU_\ell^{[1]})_{-1}^{1}\big) V_-(\ell)\big\rangle\\
&=\big\langle J^{-1}V_+(\ell),
\cM_+(\ell)(L_\ell^{[1]})_{-1}^0 A_0(\ell)\,\big(I+(L_\ell^{[0]})_0^0 A_0(\ell)\big)(L_\ell^{[1]})_{-1}^1 V_-(\ell) \big\rangle\\
&=\iD\sigma_+(\ell)\,\big\langle J^{-1}V_+(\ell),
\cM_+(\ell)(L_\ell^{[1]})_{-1}^0 A_0(\ell)^2(L_\ell^{[1]})_{-1}^1 V_-(\ell) \big\rangle
\end{align*}
since by definition $I+(L_\ell^{[0]})_0^0 A_0(\ell)=\iD\sigma_+(\ell)A_0(\ell)$. Summing the three key identities of the paragraph achieves the proof of the proposition.
\end{proof}

At this stage plugging expansions of Lemmas~\ref{lem-expa} and~\ref{lem-expG} in the formula of Proposition~\ref{p:Nformula} provides a formula for $\Gamma_\ell$ explicit in terms of $(\ell,\xi_+(\ell))$. In order to express the outcome, we find convenient to introduce a few more pieces of notation
\begin{enumerate}
\item For any function $f\in L^2(\R/T\Z)$ we denote $(f_{(m)})_{m\in\Z}$ its sequence of Fourier coefficient, normalized by $f_{(m)}:=\langle \eD^{\iD\kappa\,m\,\cdot}, f\rangle$.
\item For any $j\in \mathbb{Z},$ we denote $\omega_j(\ell):=|j\,\kappa|_{\ell,\xi_+(\ell)}\,=\,\sqrt{(j\kappa+\xi_+(\ell))^2+\ell^2}$.
\item We split $\Gamma_\ell=\Gamma_{{\rm I}}(\ell)+\Gamma_{{\rm II}}(\ell)$ with 
\begin{align*}
\Gamma_{{\rm I}}&:=-\langle J^{-1}V_+, (L^{[2]})^{1}_{-2} V_-\rangle\,,&
\Gamma_{{\rm II}}&:=-\langle J^{-1}V_+,(L^{[1]})_{-1}^0 A_0(L^{[1]})_{-1}^1 V_-\rangle\,.
\end{align*}
\end{enumerate}
We recall that $F$ denotes the extension of $r\mapsto \sqrt{r\,\tah(r)}$ from Lemma~\ref{F-aux} so that $F^2(r)=r\,\tah(r)$.

\bc\label{c:formula}
Both $\Gamma_{{\rm I}}$ and $\Gamma_{{\rm II}}$ (hence also $\ell\mapsto \Gamma_\ell$) may be extended analytically to a neighborhood of zero and
\begin{align*}
\Gamma_{{\rm I}}(\ell)
&=
(V_2)_{(-2)}\frac{F(\kappa)}{\kappa}\Big(
\kappa\big(F(\omega_1(\ell))+F(\omega_{-1}(\ell))\big)
-\xi_{+}(\ell)\big(F(\omega_1(\ell))-F(\omega_{-1}(\ell))\big)
\Big)\\
&\ 
+(a_2)_{(-2)}\frac{\tah(\kappa)}{\kappa}F(\omega_1(\ell))F(\omega_{-1}(\ell))
\ 
+(\eta_2)_{(-2)}\Big(
\kappa^2+F^2(\omega_1(\ell))F^2(\omega_{-1}(\ell))-(\ell^2+(\xi_+(\ell))^2)\Big)\\
&\ 
+(\eta_1^2)_{(-2)}
\Big(
\frac12(\kappa^2+\ell^2+(\xi_+(\ell))^2)
\big(F^2(\omega_1(\ell))+F^2(\omega_{-1}(\ell))\big)\\
&\quad\qquad\qquad
-\kappa\xi_{+}(\ell)\big(F^2(\omega_1(\ell))-F^2(\omega_{-1}(\ell))\big)
-F^2(\omega_0(\ell))F^2(\omega_1(\ell))F^2(\omega_{-1}(\ell))
\Big)
\end{align*}
\begin{align*}
&\Gamma_{{\rm II}}(\ell)
\times\Bigg(4F^2(\omega_0(\ell))-\Big(F(\omega_1(\ell))-F(\omega_{-1}(\ell))\Big)^2\Bigg)\\
&=
-\Big(
\xi_+(\ell)\,\coh(\kappa)F(\kappa)F(\omega_1(\ell))
+\sih(\kappa)\Big(\kappa\xi_+(\ell)+\ell^2+(\xi_+(\ell))^2-F^2(\omega_0(\ell))F^2(\omega_1(\ell))\Big)
\Big)\\
&\qquad
\times
\Big(
\xi_+(\ell)\,\coh(\kappa)F(\kappa)F(\omega_{-1}(\ell))
+\sih(\kappa)\Big(\kappa\xi_+(\ell)-(\ell^2+(\xi_+(\ell))^2)+F^2(\omega_0(\ell))F^2(\omega_{-1}(\ell))\Big)
\Big)\\
&
-\frac{\tah(\kappa)}{\kappa}F^2(\omega_0(\ell))
\Big(
-\Big(\kappa+\xi_+(\ell)\Big)\kappa\coh(\kappa)
+\kappa\sih(\kappa)F(\kappa)F(\omega_1(\ell))
\Big)\\
&\qquad
\times
\Big(
-\Big(\kappa-\xi_+(\ell)\Big)\kappa\coh(\kappa)
+\kappa\sih(\kappa)F(\kappa)F(\omega_{-1}(\ell))
\Big)\\
&
+\frac{F(\kappa)}{2\kappa}\Bigg(F(\omega_1(\ell))-F(\omega_{-1}(\ell))\Bigg)
\Bigg(
\Big(-(\kappa-\xi_+(\ell))\kappa\coh(\kappa)
+\kappa\sih(\kappa)F(\kappa)F(\omega_{-1}(\ell))\Big)\\
&\ \qquad\qquad
\times
\Big(
\xi_+(\ell)\,\coh(\kappa)F(\kappa)F(\omega_1(\ell))
+\sih(\kappa)\left(\kappa\xi_+(\ell)+\ell^2+(\xi_+(\ell))^2-F^2(\omega_0(\ell))F^2(\omega_1(\ell))\right)
\Big)\\
&\ \quad\qquad
+\Big(
-(\kappa+\xi_+(\ell))\kappa\coh(\kappa)
+\kappa\sih(\kappa)F(\kappa)F(\omega_1(\ell))
\Big)\\
&\ \qquad\qquad
\times
\Big(
\xi_+(\ell)\,\coh(\kappa)F(\kappa)F(\omega_{-1}(\ell))
+\sih(\kappa)\left(\kappa\xi_+(\ell)-(\ell^2+(\xi_+(\ell))^2)+F^2(\omega_0(\ell))F^2(\omega_{-1}(\ell))\right)
\Big)
\Bigg)
\end{align*}
with
\begin{align*}
(\eta_2)_{(-2)}&=\kappa\frac{\coh(2\kappa)+2}{8\tah(\kappa)}\,,&
(\eta_1^2)_{(-2)}&=\big((\eta_1)_{(-1)}\big)^2 = \frac{\sih^2(\kappa)}{4}\,,\\
(V_2)_{(-2)}&=\kappa^2\frac{2\sih^4(\kappa)+6\sih^2(\kappa)+3}{8\sih^2(\kappa)}\,,&
(a_2)_{(-2)}&=-\kappa^3\frac{\coh^2(\kappa)+5}{4\tah(\kappa)}\, .
\end{align*}
\ec

\begin{proof}
From Lemmas~\ref{lem-expa} and~\ref{lem-expG} one derives in a straightforward way the expressions for Fourier coefficients and the following relations
\begin{align*}
J^{-1}&=-J=\bp 0&1\\-1&0\ep\,,&
\mu_0&=\frac{\kappa^2}{F^2(\kappa)}\,,&
V_{\pm}(\ell)&=\bp \mp\iD \frac{F(\omega_{\mp1}(\ell))}{\sqrt{\mu_0}}\\1\ep\,,
\end{align*}
\beqs
\begin{aligned}
   & (L^{[1]})_{-1}^1=\frac12\bp
        -\iD\xi_+(\ell) \kappa \coh(\kappa) & (\ell^2-F^2(\omega_0(\ell)) F^2(\omega_1(\ell))+\xi_+(\ell)\,(\kappa+\xi_+(\ell)))\sih (\kappa)  \\[4pt]
     \kappa^2\sih (\kappa )   &  -\iD(\kappa+\xi_+(\ell))\kappa \coh(\kappa)
     \ep \\
     &  (L^{[1]})_{-1}^0=\frac12\bp
        -\iD(\xi_+(\ell)-\kappa) \kappa \coh(\kappa) & (\ell^2-F^2(\omega_0(\ell))\,F^2(\omega_{-1}(\ell))+\xi_+(\ell)(\xi_+(\ell)-\kappa))\sih (\kappa)  \\[4pt]
     \kappa^2\sih (\kappa)    &  -\iD\xi_+(\ell)\kappa \coh(\kappa)
     \ep \\
     &  (L^{[2]})_{-2}^1=\bp
        -\iD(\xi_+(\ell)-\kappa)(V_2)_{-2}  & (G^{(2)})_{(-2)}\\[4pt]
       -(a_2)_{(-2)} & -\iD(\kappa+\xi_+(\ell))(V_2)_{(-2)}
        \ep
\end{aligned}
\eeqs
and
\begin{align*}
A_0(\ell)=\frac{1}{F^2(\omega_0(\ell))-\left(\frac{F(\omega_{-1}(\ell))-F(\omega_1(\ell))}{2}\right)^2}
	\bp
        \iD\,\frac{F(\omega_{-1}(\ell))-F(\omega_1(\ell))}{2\sqrt{\mu_0}}& \frac{F^2(\omega_0(\ell))}{\mu_0}\\
   -1  & \iD\,\frac{F(\omega_{-1}(\ell))-F(\omega_1(\ell))}{2\sqrt{\mu_0}}
     \ep 
\end{align*}
where 
\begin{align*}
     (G^{(2)})_{(-2)}&= \big(-F^2(\omega_1(\ell))\,F^2(\omega_{-1}(\ell))+\ell^2+(\xi_+(\ell))^2-\kappa^2\big)\ (\eta_2)_{(-2)}\\
     &\quad + F^2(\omega_1(\ell))\,F^2(\omega_{-1}(\ell))\,F^2(\omega_0(\ell))\ ((\eta_1)_{(-1)})^2\\
     &\quad -\frac12 \big( (\ell^2 +(-\kappa+\xi_{+}(\ell))^2)F^2(\omega_1(\ell))+(\ell^2 +(\kappa+\xi_{+}(\ell))^2)F^2(\omega_{-1}(\ell)\big)\ (\eta_1^2)_{(-2)}.
\end{align*}
We only point out that when computing $A_0(\ell)$ we use that
\[
\sigma_+(\ell)
\,=\,\xi_+(\ell)+\sqrt{\mu_0}\,\frac{F(\omega_{-1}(\ell))-F(\omega_1(\ell))}{2}\,.
\]

In particular
\begin{align*}
&(L^{[1]})_{-1}^1V_-(\ell)\\
&\quad=\frac12\bp
\xi_+(\ell)\,\coh(\kappa)F(\kappa)F(\omega_1(\ell))
+\sih(\kappa)\left(\kappa\xi_+(\ell)+\ell^2+(\xi_+(\ell))^2-F^2(\omega_0(\ell))F^2(\omega_1(\ell))\right)\\
\iD\left(
-(\kappa+\xi_+(\ell))\kappa\coh(\kappa)
+\kappa\sih(\kappa)F(\kappa)F(\omega_1(\ell))
\right)
\ep\\
&-\left((L^{[1]})_{-1}^0\right)^{\textrm{T}}\,J^{-1}\overline{V_+(\ell)}\\
&\quad=\frac12\bp
\iD\left(
-(\kappa-\xi_+(\ell))\kappa\coh(\kappa)
+\kappa\sih(\kappa)F(\kappa)F(\omega_{-1}(\ell))
\right)\\
\xi_+(\ell)\,\coh(\kappa)F(\kappa)F(\omega_{-1}(\ell))
+\sih(\kappa)\left(\kappa\xi_+(\ell)-(\ell^2+(\xi_+(\ell))^2)+F^2(\omega_0(\ell))F^2(\omega_{-1}(\ell))\right)
\ep\,.
\end{align*}
From this and the above the explicit expressions stem readily.

There only remains to discuss the analyticity of $\Gamma_{{\rm II}}$ that is expressed as a quotient of two analytic functions. To prove its analyticity we only need to prove that the numerator of the quotient vanishes at $0$ at least at the same order than its denominator. On one hand $0$ is visibly at least a double root of the numerator. On the other hand Proposition~\ref{p:double} already contains that the denominator cannot vanish at a higher order as a (disguised) consequence of
\[
|(\lambda^+_{0,\ell,\xi_+(\ell)}-\iD\sigma_+(\ell))\,(\lambda^-_{0,\ell,\xi_+(\ell)}-\iD\sigma_+(\ell))|\geq\,C_0^2\,\ell^2\,.
\]
Hence the result.
\end{proof}

\subsection{Asymptotic expansion of the index}\label{ss:asymp}

In the end Theorem \ref{th:main} is deduced from Proposition~\ref{Prop:Gamma} by proving that $\Gamma_\ell$ is not zero for any sufficiently small nonzero $\ell$. As a preparation we provide some more expansions in the limit $\ell\to0$. 

\bl\label{l:asymp}
There hold
\begin{align*}
\frac{F(\omega_1(\ell))+F(\omega_{-1}(\ell))}{2}&=F(\kappa)\\
\frac{F(\omega_1(\ell))-F(\omega_{-1}(\ell))}{2}&
\stackrel{\ell\to0}{=}
\left(F'(\kappa)+(\kappa\,F''(\kappa)-F'(\kappa))\,\frac12\frac{\ell^2}{\kappa^2}\right)\,\xi_+(\ell)
+\frac16\,F'''(\kappa)\,(\xi_+(\ell))^3\,+\,\cO(\ell^5)
\end{align*}
and
\begin{align*}
\xi_+(\ell)
&\stackrel{\ell\to0}{=}\ell\,\sqrt{-\frac{F'(\kappa)}{\kappa\,F''(\kappa)}}\\
&\,\times\,\left(1
+\frac12\frac{\ell^2}{\kappa^2}\left(\frac34-\frac{F'(\kappa)}{\kappa\,F''(\kappa)}
+\frac{\kappa}{4}\left(\frac{F''(\kappa)}{F'(\kappa)}-2\frac{F'''(\kappa)}{F''(\kappa)}\right)
+\frac{\kappa}{12}\frac{F''''(\kappa)\,F'(\kappa)}{(F''(\kappa))^2}
\right)
\right)+\cO(\ell^5)\,. 
\end{align*}
\el

\begin{proof}
We find convenient to introduce
\begin{align*}
m(\ell)&:=\frac{\omega_1(\ell)+\omega_{-1}(\ell)}{2}\,,&
\delta(\ell)&:=\frac{\omega_1(\ell)-\omega_{-1}(\ell)}{2}\,.&
\end{align*}
We shall prove
\begin{align*}
m(\ell)
&\,\stackrel{\ell\to0}{=}\kappa\,
\left(1+\frac12\,\frac{\ell^2}{\kappa^2}
-\frac18\,\frac{\ell^4}{\kappa^4}
+\frac16\,\frac{\xi_+(\ell)^4}{\kappa^4}\right)
+\cO(\ell^6)\\
\delta(\ell)
&\,\stackrel{\ell\to0}{=}\xi_+(\ell)\,\left(1-\frac12\,\frac{\ell^2}{\kappa^2}\right)+\cO(\ell^5)
\end{align*}

The existence proof for $\xi_+$ already gives 
\[
\xi_+(\ell)\stackrel{\ell\to0}{=}\ell\,\sqrt{-\frac{F'(\kappa)}{\kappa\,F''(\kappa)}}+\cO(\ell^3)\,. 
\]
Then direct expansions from
\begin{align*}
\omega_{\pm1}(\ell)=\kappa\,\sqrt{1\pm2\frac{\xi_+(\ell)}{\kappa}+\left(\frac{\xi_+(\ell)}{\kappa}\right)^2+\left(\frac{\ell}{\kappa}\right)^2}
\end{align*}
yield
\begin{align*}
m(\ell)&=\frac{\omega_1(\ell)+\omega_{-1}(\ell)}{2}\\
&\,\stackrel{\ell\to0}{=}\kappa\,
\Bigg(1+\frac12\,\left(\left(\frac{\xi_+(\ell)}{\kappa}\right)^2+\left(\frac{\ell}{\kappa}\right)^2\right)
-\frac{1}{8}\left(
\left(2\frac{\xi_+(\ell)}{\kappa}\right)^2
+\left(\left(\frac{\xi_+(\ell)}{\kappa}\right)^2+\left(\frac{\ell}{\kappa}\right)^2\right)^2\right)\\
&\qquad\qquad\qquad
+\frac{3}{16}\,\left(2\frac{\xi_+(\ell)}{\kappa}\right)^2
\left(\left(\frac{\xi_+(\ell)}{\kappa}\right)^2+\left(\frac{\ell}{\kappa}\right)^2\right)
-\frac{15}{16\times24}\left(2\frac{\xi_+(\ell)}{\kappa}\right)^4\Bigg)+\cO(\ell^6)\\
&\,\stackrel{\ell\to0}{=}\kappa\,
\left(1+\frac12\,\frac{\ell^2}{\kappa^2}
-\frac18\,\frac{\ell^4}{\kappa^4}
+\frac12\,\frac{\ell^2}{\kappa^2}\,\frac{\xi^2}{\kappa^2}\right)
+\cO(\ell^6)\,.
\end{align*}
This implies
\begin{align*}
\delta(\ell)=\frac{\omega_1(\ell)-\omega_{-1}(\ell)}{2}
=\frac{\kappa\,\xi_+(\ell)}{(\omega_1(\ell)+\omega_{-1}(\ell))/2}
\,\stackrel{\ell\to0}{=}\xi_+(\ell)\,\left(1-\frac12\,\frac{\ell^2}{\kappa^2}\right)+\cO(\ell^5)\,.
\end{align*}
Now, by definition of $\xi_+(\ell)$, $\delta(\ell)$ has the same sign as $\ell$ and
\[
\frac{F(m(\ell)+\delta(\ell))+F(m(\ell)-\delta(\ell))}{2}\,=\,F(\kappa)
\]
so that a Taylor expansion yields 
\begin{align*}
\delta(\ell)
&\,\stackrel{\ell\to0}{=}\,\sign(\ell)\,\sqrt{-\frac{F(m(\ell))-F(\kappa)}{F''(m(\ell))/2}
-\frac{F''''(m(\ell))\,\delta(\ell)^4}{12\,F''(m(\ell))}+\cO(\ell^6)}\\
&\,\stackrel{\ell\to0}{=}\,\sign(\ell)\,\sqrt{-\frac{F(m(\ell))-F(\kappa)}{F''(m(\ell))/2}
-\frac{F''''(\kappa)\,(F'(\kappa))^2}{12\,(F''(\kappa))^3\kappa^2}\ell^4+\cO(\ell^6)}\,,
\end{align*}
whereas 
\begin{align*}
&\frac{F(m(\ell))-F(\kappa)}{F''(m(\ell))}
\stackrel{\ell\to0}{=}
\frac{F'(\kappa)}{F''(\kappa)}\,(m(\ell)-\kappa)
+\frac12\,(m(\ell)-\kappa)^2\,\left(1-2\frac{F'(\kappa)\,F'''(\kappa)}{(F''(\kappa))^2}\right)
+\cO(\ell^6)\\
&\qquad\quad\stackrel{\ell\to0}{=}
\frac12\frac{F'(\kappa)}{\kappa\,F''(\kappa)}\,\ell^2\left(
\,1
+\frac{\ell^2}{\kappa^2}\left(-\frac14-\frac{F'(\kappa)}{\kappa\,F''(\kappa)}
+\frac{\kappa}{4}\left(\frac{F''(\kappa)}{F'(\kappa)}-2\frac{F'''(\kappa)}{F''(\kappa)}\right)
\right)
+\cO(\ell^4)\right)\,.
\end{align*}
From this stems
\begin{align*}
\delta(\ell)
&\stackrel{\ell\to0}{=}\ell\,\sqrt{-\frac{F'(\kappa)}{\kappa\,F''(\kappa)}}\\
&\quad\times\,\left(1
+\frac12\frac{\ell^2}{\kappa^2}\left(-\frac14-\frac{F'(\kappa)}{\kappa\,F''(\kappa)}
+\frac{\kappa}{4}\left(\frac{F''(\kappa)}{F'(\kappa)}-2\frac{F'''(\kappa)}{F''(\kappa)}\right)
+\frac{\kappa}{12}\frac{F''''(\kappa)\,F'(\kappa)}{(F''(\kappa))^2}
\right)
\right)+\cO(\ell^5)\,. 
\end{align*}
The proof is then concluded by collecting the above pieces and expanding 
\[
\frac{F(\omega_1(\ell))-F(\omega_{-1}(\ell))}{2}
\stackrel{\ell\to0}{=}
F'(m(\ell))\,\delta(\ell)
+\frac16 F'''(m(\ell))\,(\delta(\ell))^3
\,+\,\cO(\ell^5)\,.
\]
\end{proof}

\begin{proposition}\label{p:asymp}
There hold
\begin{align*}
\Gamma_{{\rm I}}(0)
=&\,
2\kappa \tah(\kappa)\ (V_2)_{(-2)}
+\tah^2(\kappa)\ (a_2)_{(-2)}
+\kappa^2(1+\tah^2(\kappa))\ (\eta_2)_{(-2)}
+\kappa^3\tah(\kappa)\ (\eta_1^2)_{(-2)}\\
\Gamma_{{\rm I}}'(0)
=&\,0\\
\frac12\Gamma_{{\rm I}}''(0)
=&\,
(V_2)_{(-2)}2\frac{F(\kappa)(F'(\kappa))^2}{\kappa^2F''(\kappa)}
+(a_2)_{(-2)}\tah(\kappa)\frac{(F'(\kappa))^3}{\kappa^2F''(\kappa)}\\
&\  
-(\eta_2)_{(-2)}\left(1-\frac{F'(\kappa)}{\kappa F''(\kappa)}
-2\tah(\kappa)\frac{(F'(\kappa))^3}{F''(\kappa)}\right)\\
&\ 
+(\eta_1^2)_{(-2)}
\left(
\kappa\tah(\kappa)(1-\kappa\tah(\kappa))\left(1-\frac{F'(\kappa)}{\kappa F''(\kappa)}\right)
-\frac{\kappa(F'(\kappa))^3}{F''(\kappa)}
+4\frac{F(\kappa)(F'(\kappa))^2}{F''(\kappa)}
\right)
\end{align*}
with Fourier coefficients given in Corollary~\ref{c:formula}, and
\begin{align*}
\Gamma_{{\rm II}}(0)
&=\frac{\kappa^2}{4}\frac{
4\sih^2(\kappa) F'(\kappa)
+4\tah(\kappa)F(\kappa)(F'(\kappa))^2
-\frac{\tah(\kappa)}{\coh^2(\kappa)}\kappa(\kappa F''(\kappa)-F'(\kappa))
}{\kappa F''(\kappa)-F'(\kappa)
+\,(F'(\kappa))^3}\\
\Gamma_{{\rm II}}'(0)
&=0
\end{align*}
\begin{align*}
\hspace{-2em}
\frac12\Gamma_{{\rm II}}''&(0)
\times 4\,\left(\kappa F''(\kappa)-F'(\kappa)
+\,(F'(\kappa))^3\right)\\
&=\frac{4\Gamma_{{\rm II}}(0)}{\kappa F''(\kappa)}
\Bigg(
\frac13(\kappa F''(\kappa)-F'(\kappa))^2
+\,(F'(\kappa))^2
\left(-\frac{F''(\kappa)}{\kappa}(\kappa\,F''(\kappa)-F'(\kappa))\,
+\frac13\,F'''(\kappa)\,F'(\kappa)
\right)
\Bigg)\\
&\ 
+\Gamma_{{\rm II}}(0)\,(1-(F'(\kappa))^2)\frac{F'(\kappa)}{\kappa^2}
\left(3-4\frac{F'(\kappa)}{\kappa\,F''(\kappa)}
+\kappa\left(\frac{F''(\kappa)}{F'(\kappa)}-2\frac{F'''(\kappa)}{F''(\kappa)}\right)
+\frac{\kappa\,F''''(\kappa)\,F'(\kappa)}{3\,(F''(\kappa))^2}
\right)\\
&\ 
+\frac{F'(\kappa)}{4}
\left(
4\sih^2(\kappa)
+\kappa\frac{\tah(\kappa)}{\coh^2(\kappa)}
+4\tah(\kappa)F(\kappa)F'(\kappa)
\right)\\
&\qquad\quad\quad\times
\left(3-4\frac{F'(\kappa)}{\kappa\,F''(\kappa)}
+\kappa\left(\frac{F''(\kappa)}{F'(\kappa)}-2\frac{F'''(\kappa)}{F''(\kappa)}\right)
+\frac{\kappa\,F''''(\kappa)\,F'(\kappa)}{3\,(F''(\kappa))^2}
\right)\\
&\ 
+\frac{1}{\kappa F''(\kappa)}\Bigg(
\left(\frac13\kappa^3\frac{\tah(\kappa)}{\coh^2(\kappa)}
+\sih^2(\kappa)(1-\kappa\tah(\kappa))^2
\right)\,(\kappa F''(\kappa)-F'(\kappa))^2\\
&\qquad\quad\quad
+2\kappa^2\tah(\kappa)F(\kappa)\,F'(\kappa)\,
\left((\kappa\,F''(\kappa)-F'(\kappa))\,\frac{F''(\kappa)}{\kappa}
-\frac13\,F'''(\kappa)\,F'(\kappa)
\right)
\\
&\qquad\quad\quad
-F'(\kappa)(\kappa\,F''(\kappa)-F'(\kappa))\kappa\coh(\kappa)\sih(\kappa)\\
&\qquad\quad\quad 
-(\kappa\,F''(\kappa)-F'(\kappa)) F(\kappa)(F'(\kappa))^2
\,2\sih(\kappa)(\kappa\sih(\kappa)+2\coh(\kappa))\\
&\qquad\quad\quad
-(\kappa\,F''(\kappa)-F'(\kappa))(F'(\kappa))^3
\kappa\tah(\kappa)(\kappa\tah(\kappa)(3\sih^2(\kappa)+4)-2\sih^2(\kappa))\\
&\qquad\quad\quad
+(F'(\kappa))^4\,3\kappa\sih(\kappa)\coh(\kappa)
-(F'(\kappa))^5\,2\kappa\sih^2(\kappa)F(\kappa)\Bigg)\,.
\end{align*}
\end{proposition}

We stress that much less computations would be required if examining the value $\Gamma_0$ (of $\Gamma$ at $\ell=0$) were sufficient. Unfortunately there is a value of $\kappa$ at which this value vanishes and a higher-order computation is needed.

\begin{proof}
In order to use the foregoing lemma we observe that
\begin{align*}
F^2(\omega_1(\ell))-F^2(\omega_{-1}(\ell))
&=
4\frac{F(\omega_1(\ell))+F(\omega_{-1}(\ell))}{2}\,\frac{F(\omega_1(\ell))-F(\omega_{-1}(\ell))}{2}\\
F(\omega_1(\ell))F(\omega_{-1}(\ell))
&=
\left(\frac{F(\omega_1(\ell))+F(\omega_{-1}(\ell))}{2}\right)^2
-\,\left(\frac{F(\omega_1(\ell))-F(\omega_{-1}(\ell))}{2}\right)^2\\
F^2(\omega_1(\ell))+F^2(\omega_{-1}(\ell))
&=
2\,\left(\frac{F(\omega_1(\ell))+F(\omega_{-1}(\ell))}{2}\right)^2
+2\,\left(\frac{F(\omega_1(\ell))-F(\omega_{-1}(\ell))}{2}\right)^2\,.
\end{align*}

Then we derive 
\begin{align*}
\Gamma_{{\rm I}}(\ell)
\stackrel{\ell\to0}{=}&
(V_2)_{(-2)}\frac{F(\kappa)}{\kappa}\Big(
2\kappa F(\kappa)
-2(\xi_{+}(\ell))^2\,F'(\kappa)\,
\Big)\\
&\ 
+(a_2)_{(-2)}\frac{\tah(\kappa)}{\kappa}\Big(\kappa\tah(\kappa)-(F'(\kappa))^2(\xi_{+}(\ell))^2\Big)\\
&\  
+(\eta_2)_{(-2)}\Big(
\kappa^2(1+\tah^2(\kappa))-2\kappa\tah(\kappa)(F'(\kappa))^2(\xi_{+}(\ell))^2-(\ell^2+(\xi_+(\ell))^2)\Big)\\
&\ 
+(\eta_1^2)_{(-2)}
\Big(
\kappa\tah(\kappa)(\kappa^2+\ell^2+(\xi_+(\ell))^2)
+\kappa^2(F'(\kappa))^2(\xi_{+}(\ell))^2\\
&\quad\qquad\qquad
-4\kappa F(\kappa)F'(\kappa)\,(\xi_{+}(\ell))^2
-\kappa^2\tah^2(\kappa)\,(\ell^2+(\xi_+(\ell))^2)
\Big)+\cO(\ell^3)
\end{align*}
\begin{align*}
4F^2(\omega_0(\ell))&-\Big(F(\omega_1(\ell))-F(\omega_{-1}(\ell))\Big)^2\\
\stackrel{\ell\to0}{=}&
4\,(\ell^2+(\xi_+(\ell))^2)-4\,(F'(\kappa))^2\,(\xi_{+}(\ell))^2
-\frac43\,(\ell^2+(\xi_+(\ell))^2)^2\\
&
-4\,F'(\kappa)\,(\xi_+(\ell))^2
\left((\kappa\,F''(\kappa)-F'(\kappa))\,\frac{\ell^2}{\kappa^2}
+\frac13\,F'''(\kappa)\,(\xi_+(\ell))^2
\right)+\cO(\ell^5)
\end{align*}
and
\begin{align*}
&\Gamma_{{\rm II}}(\ell)
\times\Bigg(4F^2(\omega_0(\ell))-\Big(F(\omega_1(\ell))-F(\omega_{-1}(\ell))\Big)^2\Bigg)\\
&\stackrel{\ell\to0}{=}
-\kappa^3\frac{\tah(\kappa)}{\coh^2(\kappa)}F^2(\omega_0(\ell))
+(\xi_+(\ell))^2
\left(-4\kappa^2\sih^2(\kappa)\right)
-4\kappa^2\tah(\kappa)F(\kappa)\frac{F(\omega_1(\ell))-F(\omega_{-1}(\ell))}{2}\xi_+(\ell)\\
&+\left(\ell^2+(\xi_+(\ell))^2-F^2(\omega_0(\ell))\kappa\tah(\kappa)
\right)^2\sih^2(\kappa)
+(\xi_+(\ell))^2
F^2(\omega_0(\ell))\kappa\coh(\kappa)\sih(\kappa)\\
&+\frac{F(\omega_1(\ell))-F(\omega_{-1}(\ell))}{2}\xi_+(\ell)
\left(F^2(\omega_0(\ell))
2\kappa\sih^2(\kappa)F(\kappa)
+4\coh(\kappa)\sih(\kappa)F(\kappa)\left(\ell^2+(\xi_+(\ell))^2\right)
\right)\\
&+\left(\frac{F(\omega_1(\ell))-F(\omega_{-1}(\ell))}{2}\right)^2
\left(F^2(\omega_0(\ell))\kappa^2\tah^2(\kappa)(3\sih^2(\kappa)+4)
-2\kappa\sih^2(\kappa)\tah(\kappa)\left(\ell^2+(\xi_+(\ell))^2\right)\right)\\
&+\left(\frac{F(\omega_1(\ell))-F(\omega_{-1}(\ell))}{2}\right)^2(\xi_+(\ell))^2
\,3\kappa\sih(\kappa)\coh(\kappa)\\
&-\left(\frac{F(\omega_1(\ell))-F(\omega_{-1}(\ell))}{2}\right)^3\xi_+(\ell)
\,2\kappa\sih^2(\kappa)F(\kappa)
+\cO(\ell^6)
\end{align*}
so that
\begin{align*}
\hspace{-2em}
\Gamma_{{\rm II}}&(\ell)
\times\Bigg(4F^2(\omega_0(\ell))-\Big(F(\omega_1(\ell))-F(\omega_{-1}(\ell))\Big)^2\Bigg)\\
&\stackrel{\ell\to0}{=}
-4\kappa^2\sih^2(\kappa)\,(\xi_+(\ell))^2
-\kappa^3\frac{\tah(\kappa)}{\coh^2(\kappa)}\,(\ell^2+(\xi_+(\ell))^2)
-4\kappa^2\tah(\kappa)F(\kappa)F'(\kappa)\,(\xi_+(\ell))^2\\
&\ 
+\frac13\kappa^3\frac{\tah(\kappa)}{\coh^2(\kappa)}\,(\ell^2+(\xi_+(\ell))^2)^2
-2\kappa^2\tah(\kappa)F(\kappa)\,
(\xi_+(\ell))^2\,
\left((\kappa\,F''(\kappa)-F'(\kappa))\,\frac{\ell^2}{\kappa^2}
+\frac13\,F'''(\kappa)\,(\xi_+(\ell))^2
\right)
\\
&\ 
+\sih^2(\kappa)(1-\kappa\tah(\kappa))^2\,\,(\ell^2+(\xi_+(\ell))^2)^2
+(\xi_+(\ell))^2(\ell^2+(\xi_+(\ell))^2)\kappa\coh(\kappa)\sih(\kappa)\\
&\ 
+(\xi_+(\ell))^2(\ell^2+(\xi_+(\ell))^2) F(\kappa)F'(\kappa)
\,2\sih(\kappa)(\kappa\sih(\kappa)+2\coh(\kappa))\\
&\ 
+(\xi_+(\ell))^2(\ell^2+(\xi_+(\ell))^2)(F'(\kappa))^2
\kappa\tah(\kappa)(\kappa\tah(\kappa)(3\sih^2(\kappa)+4)-2\sih^2(\kappa))\\
&\ 
+(\xi_+(\ell))^4 (F'(\kappa))^2\,3\kappa\sih(\kappa)\coh(\kappa)
-(\xi_+(\ell))^4 (F'(\kappa))^3\,2\kappa\sih^2(\kappa)F(\kappa)
\ +\ \cO(\ell^5)\,.
\end{align*}

The final formulas are then obtained by expanding $\xi_+(\ell)$ and collecting terms.
\end{proof}

Parts of the verification of the non vanishing of the index is completed with validated numerics resulting in a computer-assisted proof. To prepare the latter we provide alternative formulas for Proposition~\ref{p:asymp}, that are cumbersome but slightly more explicit.

\bc\label{c:more}
There hold
\begin{align*}
\Gamma_{{\rm I}}(0)
&=\frac{\kappa^3}
{8\coh(\kappa)\sih(\kappa)}
\Big(8\sih^4(\kappa)+8\sih^2(\kappa)+9\Big)\\
\Gamma_{{\rm II}}(0)
&=-\frac{\kappa^3}{2}\tah(\kappa)
\times\Big(\kappa^3(1+4\sih^2(\kappa))
+4\kappa^2\coh(\kappa)\sih(\kappa)
+\kappa\sih^2(\kappa)(1+\sih^2(\kappa))\left(8\sih^2(\kappa)+19\right)\\
&\qquad\qquad\qquad\qquad
+4\coh(\kappa)\sih^3(\kappa)(1+\sih^2(\kappa))(3+2\sih^2(\kappa))
\Big)\\
&\qquad\times\Big(
\kappa^3(8\sih^4(\kappa)+10\sih^2(\kappa)+1)
-3\kappa^2\coh(\kappa)\sih(\kappa)\\
&\qquad\qquad+3\kappa\sih^2(\kappa)(1+\sih^2(\kappa))(2\sih^2(\kappa)+1)
-\coh(\kappa)\sih^3(\kappa)(1+\sih^2(\kappa))
\Big)^{-1}\,.
\end{align*}
Therefore
\begin{align*}
\lim_{\ell\to0}\Gamma_\ell
&=\frac{\kappa^3}
{8\coh(\kappa)\sih(\kappa)}
\times\Big(
\kappa^3(8\sih^4(\kappa)+10\sih^2(\kappa)+1)
-3\kappa^2\coh(\kappa)\sih(\kappa)\\
&\qquad\qquad\qquad
+3\kappa\sih^2(\kappa)(1+\sih^2(\kappa))(2\sih^2(\kappa)+1)
-\coh(\kappa)\sih^3(\kappa)(1+\sih^2(\kappa))
\Big)^{-1}\\
&\qquad\times\Big(
\kappa^3\left(64\sih^8(\kappa)+144\sih^6(\kappa)+144\sih^4(\kappa)+94\sih^2(\kappa)+9\right)
\\
&\qquad\qquad
-\kappa^2\coh(\kappa)\sih(\kappa)\left(24\sih^4(\kappa)+40\sih^2(\kappa)+27\right)
\\
&\qquad\qquad
+\kappa\sih^2(\kappa)(1+\sih^2(\kappa))\left(48\sih^6(\kappa)+40\sih^4(\kappa)+2\sih^2(\kappa)+27\right)
\\
&\qquad\qquad
-\coh(\kappa)\sih^3(\kappa)(1+\sih^2(\kappa)))
\left(40\sih^4(\kappa)+56\sih^2(\kappa)+9\right)
\Big)\,.
\end{align*}
\ec

Similar formulas for $\Gamma_{{\rm I}}''(0)$ and $\Gamma_{{\rm II}}''(0)$ are given in Appendix~\ref{additional_formula}.

\begin{proof}
To make explicit the $\Gamma_{{\rm I}}(0)$, one simply needs to replace Fourier coefficients with their explicit formulas and expand.

Concerning the $\Gamma_{{\rm II}}(0)$ part, we first need to give explicit formulas for derivatives of $F$. Note that
\begin{align*}
\frac{F'(\kappa)}{F(\kappa)}
&=\frac12\frac{1}{\kappa}+\frac12\left(\frac{1}{\tah(\kappa)}-\tah(\kappa)\right)
=\frac12\frac{1}{\kappa}+\frac12\frac{1}{\coh(\kappa)\sih(\kappa)}
=\frac{\kappa+\coh(\kappa)\sih(\kappa)}{2\kappa\coh(\kappa)\sih(\kappa)}\\
\frac{F''(\kappa)}{F(\kappa)}
&=\left(\frac{F'}{F}\right)'(\kappa)
+\left(\frac{F'(\kappa)}{F(\kappa)}\right)^2
=-\frac12\frac{1}{\kappa^2}+\frac12\left(-\frac{1}{\tah^2(\kappa)}+\tah^2(\kappa)\right)
+\left(\frac{F'(\kappa)}{F(\kappa)}\right)^2\\
&
=-\frac12\frac{1}{\kappa^2}+\left(\frac{F'(\kappa)}{F(\kappa)}\right)^2
-\frac12\frac{1+2\sih^2(\kappa)}{\coh^2(\kappa)\sih^2(\kappa)}\\
&=-\frac{\kappa^2(1+4\sih^2(\kappa))-2\kappa\coh(\kappa)\sih(\kappa)+\sih^2(\kappa)(1+\sih^2(\kappa))
}{4\kappa^2\coh^2(\kappa)\sih^2(\kappa)}\\
\frac{\kappa F''(\kappa)-F'(\kappa)}{F(\kappa)}
&
=-\frac{\kappa^2(1+4\sih^2(\kappa))+3\sih^2(\kappa)(1+\sih^2(\kappa))
}{4\kappa\coh^2(\kappa)\sih^2(\kappa)}\,.
\end{align*}
Thus
\begin{align*}
&\kappa \frac{F''(\kappa)}{F(\kappa)}-\frac{F'(\kappa)}{F(\kappa)}
+\,\kappa\tah(\kappa)\,\left(\frac{F'(\kappa)}{F(\kappa)}\right)^3\\
&=-\Bigg(\kappa^3(8\sih^4(\kappa)+10\sih^2(\kappa)+1)
-3\kappa^2\coh(\kappa)\sih(\kappa)\\
&\qquad+3\kappa\sih^2(\kappa)(2\sih^4(\kappa)+3\sih^2(\kappa)+1)
-\coh(\kappa)\sih^3(\kappa)(1+\sih^2(\kappa))
\Bigg)
/(8\kappa^2\coh^4(\kappa)\sih^2(\kappa))
\end{align*}
and
\begin{align*}
&4\sih^2(\kappa)\frac{F'(\kappa)}{F(\kappa)}
+4\tah^2(\kappa)\kappa\left(\frac{F'(\kappa)}{F(\kappa)}\right)^2
-\frac{\tah(\kappa)}{\coh^2(\kappa)}\kappa\left(\kappa \frac{F''(\kappa)}{F(\kappa)}-\frac{F'(\kappa)}{F(\kappa)}\right)\\
&=\Bigg(\kappa^3(1+4\sih^2(\kappa))
+4\kappa^2\coh(\kappa)\sih(\kappa)
+\kappa\sih^2(\kappa)(1+\sih^2(\kappa))\left(8\sih^2(\kappa)+19\right)\\
&\qquad
+4\coh(\kappa)\sih^3(\kappa)(1+\sih^2(\kappa))(3+2\sih^2(\kappa))
\Bigg)/(4\kappa\coh^5(\kappa)\sih(\kappa))
\end{align*}
which lead to the formula for $\Gamma_{{\rm II}}(0)$.

\end{proof}

\subsection{Validated numerics}

Since $\kappa$ plays a role in the final discussion we restore now marks of the dependence on $\kappa$. In particular we denote $\Gamma(\kappa,\ell)$ for $\Gamma_\ell$ and its analytic extension. Note however that we have not proved joint regularity for $\Gamma$.

We now complete the proof of the following theorem.

\begin{theorem}\label{th:final}
For any $\kappa>0$, there exists $\ell_*(\kappa)>0$ such that for any $0<|\ell|\leq \ell_*(\kappa)$, there exist positive $\eps_*(\kappa,\ell)$ and $c_*(\kappa,\ell)$ such that for any $0<|\eps|\leq \eps_*(\kappa,\ell)$ the Bloch symbol $L^\eps_{\ell,\xi_\ell(\eps)}$ (with $\xi_\ell(\eps)$ as in Lemma~\ref{l:curve}) possesses an eigenvalue of real part larger than $c_*(\kappa,\ell)\,\eps^2$. 
\end{theorem}

\begin{remark}\label{rk:shape}
Going back to \eqref{e:char} one may derive from our analysis a more precise description of the unstable spectrum than stated in Theorem~\ref{th:final} and even obtain the expected ellipsoid type shape of the unstable spectrum at fixed $\ell$ with diameters of size $\cO(\eps^2)$. Indeed, focusing on the generic case when $\Gamma(\kappa,0)\neq0$ the outcome of the latter is that for any $\kappa>0$, there exist positive $\mathfrak{c}_0$, $\mathfrak{K}_0$, $\mathfrak{e}_0$ and $\mathfrak{l}_0$ such that for any $0<\eps\leq \mathfrak{e}_0$, for any $(\ell,\xi)$ such that $\mathfrak{K}_0\eps\leq \ell\leq \mathfrak{l}_0$ and $|\xi-\xi_+(\ell)|\leq \mathfrak{c}_0|\ell|$ the spectrum of $L^\eps_{\ell,\xi}$ in the ball $B(\iD\sigma_+(\ell), \mathfrak{c}_0|\ell|)$ is given by
\begin{align}\label{eq:char}
\left(\frac{\lambda}{\iD}-\mathfrak{A}(\ell,\xi,\eps)\right)^2
\,=\,\left(\mathfrak{B}(\ell,\xi,\eps)\right)^2-\left(\mathfrak{C}(\ell,\xi,\eps)\right)^2
\end{align}
with $\mathfrak{A}$, $\mathfrak{B}$, $\mathfrak{C}$ real-valued,
\begin{align*}
\mathfrak{A}&:=\frac{1}{4}\left(-\frac{b_+}{\alpha_+}
+\frac{b_-}{\alpha_-}\right)
\,=\,\frac{\lambda^-_{1,\ell,\xi}+\lambda^+_{-1,\ell,\xi}}{2\iD}+\cO(\eps)
\,=\,
\sigma_+(\ell)+\cO(\eps+|\xi-\xi_+(\ell)|)\\
\mathfrak{B}&:=\frac14\left(\frac{b_+}{\alpha_+}
+\frac{b_-}{\alpha_-}\right)
\,=\,\left(-\frac12\d_\xi^2\varphi(0,0)\,\xi_+(\ell)+\cO(\eps+|\xi-\xi_+(\ell)|)\right)\,(\xi-\xi_\ell(\eps))\\
\mathfrak{C}&:=\frac12\frac{|a|}{\sqrt{\alpha_-\alpha_+}}
\,=\,\frac{\eps^2}{2\sqrt{\alpha_-\alpha_+}}\left(|\Gamma_\ell|+\cO(\eps+|\xi-\xi_+(\ell)|)\right)\,.
\end{align*}
Note that the unstable part of the spectrum we describe corresponds to the parameters such that the right-hand side of \eqref{eq:char} is negative. When $(\eps,\ell)$ is held fixed and $\xi$ is varied over the full interval of length $\cO(|\ell|)$, this right-hand side changes sign twice, the unstable zone being covered over an $\cO(\eps^2/|\ell|)$-neighborhood of $\xi_\ell(\eps)$. At fixed $(\eps,\ell)$ the shape of the spectrum we describe is given by the union of 
\begin{itemize}
\item an ellipsoid-like curve with $\cO(\eps^2)$ variation in real parts and $\cO(\eps^2/|\ell|)$ variation in imaginary parts
\item two segments of the imaginary axis (covered with variable multiplicity), crossing the ellipsoid part, each of $\cO(|\ell|)$ length.
\end{itemize}
See \cite{NRS-EP} for related illustrating numerical solutions and Figure~\ref{figcartoon} for a sketch. Fixing $\eps$ and varying $(\ell,\xi)$ the spectrum covers an $\cO(\eps^2)$-neighborhood of an $\cO(1)$ portion of the imaginary axis, staying $\cO(|\eps|)$ away from the origin.
\end{remark}

\begin{figure}[htbp]
\centering
\includegraphics[scale=0.4]{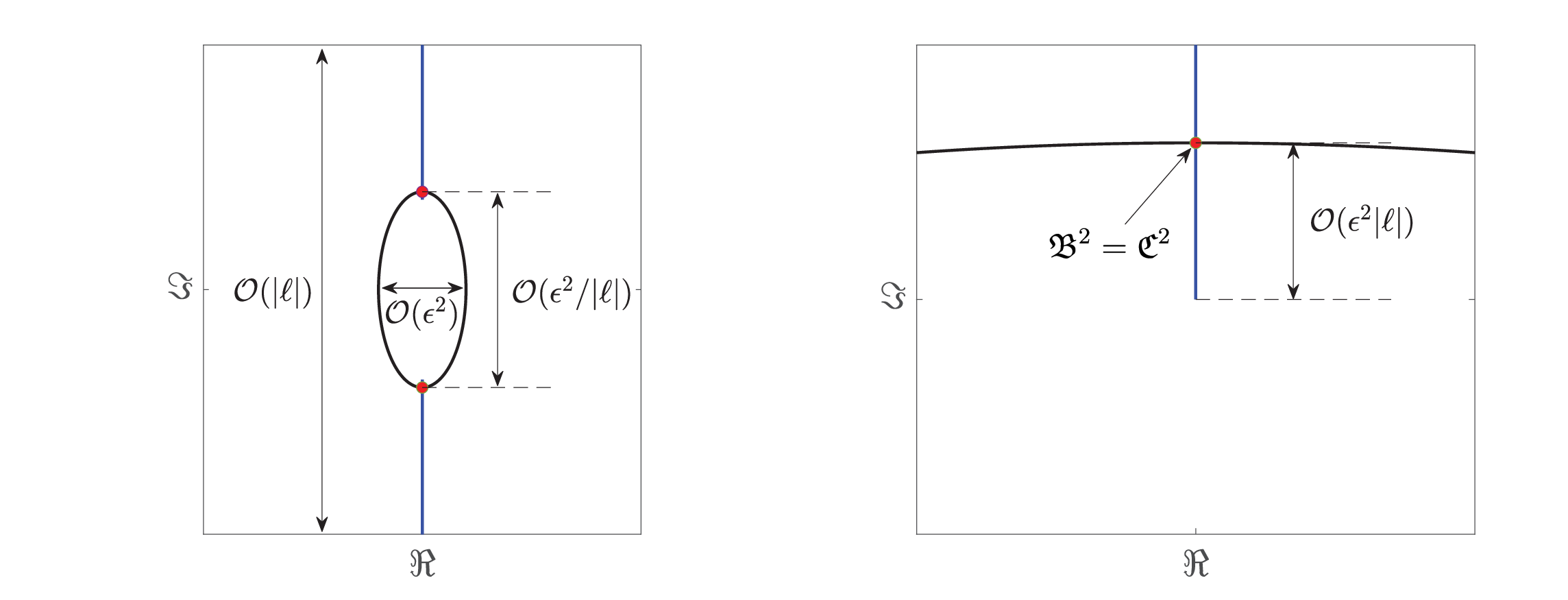}
\caption{Cartoon of spectrum for fixed $(\eps,\ell)$.}
\label{figcartoon}
\end{figure}

\begin{corollary}\label{c:linear}
For some $K>0$ and any $\kappa>0$, there exist positive $c_{**}(\kappa)$ and $\eps_{**}(\kappa)$ such that for any $0<|\eps|\leq \eps_{**}(\kappa)$ there exists a nonzero $U^\eps\in L^2(\R^2)\times H^{1/2}(\R^2)$ satisfying for any $t\geq0$ 
\[
\|\eD^{t\,L^\eps}U^\eps\|_{L^2(\R^2)\times H^{1/2}(\R^2)}
\geq K\,\eD^{t\,c_{**}(\kappa)\,\eps^2}\,\|U^\eps\|_{L^2(\R^2)\times H^{1/2}(\R^2)}
\]
where $(\eD^{t\,L^\eps})_{t\geq0}$ denotes the semigroup generated by $L^\eps$.
\end{corollary}
Let us begin by deducing Corollary~\ref{c:linear} from Theorem~\ref{th:final}. 
\begin{proof}
Let $\kappa>0$ and, with notation from Theorem~\ref{th:final}, set $\eps_{**}(\kappa):=\eps_*(\kappa,\ell_*(\kappa))$ and $c_{**}(\kappa):=c_*(\kappa,\ell_*(\kappa))/2$. Let $0<|\eps|\leq \eps_{**}(\kappa)$. We deduce from Theorem~\ref{th:final} (and its proof) that there exists $\lambda_*$ (depending on $\eps$ and $\kappa$) such that $\Re(\lambda_*)\geq 2\,c_{**}(\kappa)\,\eps^2$ and $\lambda_*$ is a simple\footnote{We use this element to pick a continuous family of eigenvectors and simplify the proof but it is necessary since a measurable choice of eigenvectors and is always available, see \cite[p.30-31]{R}.} eigenvalue of $L^\eps_{\ell_*(\kappa),\xi_{\ell_*(\kappa)}(\eps)}$. Then there exists a positive $\delta_*$ (depending on $\eps$ and $\kappa$) such that when $|\ell-\ell_*(\kappa)|\leq \delta_*$ and $|\xi-\xi_{\ell_*(\kappa)}(\eps)|\leq \delta_*$, $L^\eps_{\ell,\xi}$ possesses a simple eigenvalue of real part larger than $c_{**}(\kappa)\,\eps^2$ and one may choose the eigenvalue and a corresponding eigenvector $(\lambda^\eps_{\ell,\xi},U^\eps_{\ell,\xi})$ continuously with respect to $(\ell,\xi)$. For instance one may obtain the eigenvector by applying a suitable spectral projector of rank $1$ on an eigenvector assoctiated with $\lambda_*$ for $L^\eps_{\ell_*(\kappa),\xi_{\ell_*(\kappa)}(\eps)}$. Finally one may set
\[
U^\eps(x,y)\ :=\ \int_{\xi_{\ell_*(\kappa)}(\eps)-\delta_*}^{\xi_{\ell_*(\kappa)}(\eps)+\delta_*}\int_{\ell_*(\kappa)-\delta_*}^{\ell_*(\kappa)+\delta_*} \eD^{\iD\xi x+\iD\ell y}\ U^\eps_{\ell,\xi}(x)\ \dif\ell\,\dif\xi\,,
\]
so that
\[
(\eD^{t\,L^\eps}\,U^\eps)(x,y)\ =\ \int_{\xi_{\ell_*(\kappa)}(\eps)-\delta_*}^{\xi_{\ell_*(\kappa)}(\eps)+\delta_*}\int_{\ell_*(\kappa)-\delta_*}^{\ell_*(\kappa)+\delta_*} \eD^{\lambda^\eps_{\ell,\xi}\,t}\eD^{\iD\xi x+\iD\ell y}\ U^\eps_{\ell,\xi}(x)\ \dif\ell\,\dif\xi\,.
\]
The constant $K$ simply quantifies the equivalence of norms in spatial and Bloch sides.
\end{proof}
We now turn to proving Theorem~\ref{th:final}. In view of Proposition~\ref{Prop:Gamma} it is sufficient to prove that, for any $\kappa>0$, there is a neighborhood of zero on which $\Gamma(\kappa,\cdot)$ may vanish only at zero. Since $\d_\ell \Gamma(\kappa,0)$ is always $0$ it is sufficient to prove the following proposition.

\begin{proposition}\label{p:numerics}
For any $\kappa>0$, $(\Gamma(\kappa,0)),\d_\ell^2 \Gamma(\kappa,0))\neq(0,0)$. 
\end{proposition}

We have used interval arithmetics to obtain a computer-assisted proof of Proposition~\ref{p:numerics}. Note that the functions we need to evaluate are rational functions of the exponential function so that no extra fancy work is needed here to control remainders in Taylor expansions. We may directly use the interval arithmetic MATLAB package INTLAB and its accompanying bounds for truncation and rounding errors \cite{Ru99a}. Again we refer to \cite{Gomez-Serrano} for a brief introduction to the underlying techniques and a survey of some applications to free-boundary incompressible flows. 

We prove in this way the following lemmas that imply Proposition~\ref{p:numerics} thus Theorem~\ref{th:final}. The codes used to derive the computer-assisted parts of the proof is available at \cite{jiao_2024_13627082}.

\begin{lemma}
The function $\Gamma(\cdot,0)$ does not vanish on $\R_+^*\setminus I_*$ where $I_*:=(0.380337,0.380338)$.
\end{lemma}

\begin{proof}
We first recast $\Gamma(\cdot,0)$ as
\[
\Gamma(\kappa,0)=\frac{\kappa^3}{4\eD^{2\kappa}(\eD^{4\kappa} - 1)}\frac{\gamma_{{\rm top}}(\kappa)}{\gamma_{{\rm bot}}(\kappa)}
\] 
where 
\begin{align*}
\gamma_{{\rm top}}(\kappa):=&\,(6\kappa - 5)\eD^{20\kappa}+8(4\kappa^3 - 2\kappa - 1)\eD^{18\kappa}+(32\kappa^3 - 48\kappa^2 + 2\kappa + 23)\eD^{16\kappa}\\&\,+16(20\kappa^3 - 8\kappa^2 + 15\kappa + 1)\eD^{14\kappa}+4(232\kappa^3 - 116\kappa^2 - 2\kappa - 11)\eD^{12\kappa}\\
&\,-64\kappa (23\kappa^2 + 7)\eD^{10\kappa}+4(232\kappa^3 + 116\kappa^2 - 2\kappa + 11)\eD^{8\kappa}\\
&\,+16(20\kappa^3 + 8\kappa^2 + 15\kappa - 1)\eD^{6\kappa}+(32\kappa^3 + 48\kappa^2 + 2\kappa - 23)\eD^{4\kappa}\\
&\,+8(4\kappa^3 - 2\kappa + 1)\eD^{2\kappa}+6\kappa + 5\,.
\end{align*}
and $\gamma_{{\rm bot}}$ does not vanish\footnote{As already pointed out in the proof of Corollary~\ref{c:formula}, this is a consequence of Proposition~\ref{p:double}.} on $\R_+^*$. To begin with we observe that when $\kappa\geq2$ 
\begin{align*}
\gamma_{{\rm top}}(\kappa)\geq&\,\frac72\kappa\eD^{20\kappa}
+27\kappa^3\eD^{18\kappa}
+(8\kappa^3 + 2\kappa + 23)\eD^{16\kappa}\\
&\,+16(16\kappa^3+ 15\kappa + 1)\eD^{14\kappa}
+4\left(174-\frac{15}{8}\right)\kappa^3\eD^{12\kappa}\\
&\,-64\kappa (23\kappa^2 + 7)\eD^{10\kappa}
+4(232\kappa^3 + 115\kappa^2+ 11)\eD^{8\kappa}\\
&\,+16(20\kappa^3 + 8\kappa^2 +\frac{29}{2}\kappa)\eD^{6\kappa}
+\left(32\kappa^3 + \left(48-\frac{19}{4}\right)\kappa^2\right)\eD^{4\kappa}\\
&\,+8\left(\frac72\kappa^3+ 1\right)\eD^{2\kappa}+6\kappa + 5\\
\geq& 
27\kappa^3\eD^{18\kappa}
\,-64\kappa (23\kappa^2 + 7)\eD^{10\kappa}
\geq \eD^{10\kappa}\kappa^3\left(27\eD^{16}-64\left(23+\frac74\right)\right)>0\,.
\end{align*}

Now we turn to the small $\kappa$ regime. Direct (but lengthy) Taylor expansions show that $\gamma_{{\rm top}}^{(j)}(0)=0$ for $0\leq j\leq 6$. Then suitable interval arithmetic computations prove that $\gamma_{{\rm top}}^{(7)}$ does not vanish on $[0,0.15]$ so that, by integral Taylor formula, $\gamma_{{\rm top}}$ does not vanish on $[0,0.15]$ either. To be more explicit on computations we point out that we have proved signs by appealing to INTLAB on an equally spaced subdivision of $[0,0.15]$ with $15000$ subintervals. For instance one proves in this way that
\[
\gamma_{{\rm top}}^{(7)}([0,0.00001])
\subset 10^{7} \times[-1.30241843333415,  -0.76279327605071]
\]
and
\[
\gamma_{{\rm top}}^{(7)}([0.14999,0.15])
\subset 10^8 \times [-4.15360258309128,-3.32398859588027] \,.
\]

We have concluded the proof by checking with INTLAB that $\gamma_{{\rm top}}$ is negative on $[0.15,0.380337]$ and positive on $[0.380338,2]$. Again this is proved by appealing to INTLAB on a finite but very large number of subintervals.  
\end{proof}

Though it is not useful to prove Theorem~\ref{th:final} we point out that we have also checked that $\gamma_{{\rm top}}'$ does not vanish on $I_*$ so that the zero of $\Gamma(\cdot,0)$ in $I_*$ is a simple zero.

\begin{lemma}
The function $\d_\ell^2\Gamma(\cdot,0)$ does not vanish on $I_*=(0.380337,0.380338)$.
\end{lemma}

\begin{proof}
A single INTLAB computation is sufficient to establish that
\[
\d_\ell^2\Gamma(\cdot,0)(\,\overline{I_*}\,)\subset
[ -0.84878324244033,  -0.80979638157683] \,.
\]
To carry out this computation we have used more explicit formulas of $\d_\ell^2\Gamma(\cdot,0)$ that we provide in an appendix.
\end{proof}

\begin{appendix}
\section{Additional formulas for $\d_\ell^2\Gamma(\cdot,0)$}\label{additional_formula}

As in the derivation of the formula for $\Gamma(\cdot,0)$ in Corollary~\ref{c:more}, by expanding Fourier coefficients and derivatives of $F$ in formulas from Proposition~\ref{p:asymp}, we obtain the following formulas used to evaluate $\d_\ell^2\Gamma(\cdot,0)=\Gamma_{{\rm I}}''(0)+\Gamma_{{\rm II}}''(0)$
\begin{align*}
\Gamma_{{\rm I}}''(0)=&\,\frac{\kappa}{8\eD^{2\kappa}(\eD^{2 \kappa }-1)(\eD^{2 \kappa }+1)^2}
\big((16 \eD^{2 \kappa } (\eD^{4 \kappa }-\eD^{2 \kappa }+1))\kappa ^2-8 \eD^{2 \kappa } (\eD^{4 \kappa }-1)\kappa+(\eD^{4 \kappa }-1)^2\big)^{-1}\\
&\times\Big(-16 \eD^{2 \kappa } (6 \eD^{2 \kappa }-16 \eD^{4 \kappa }+49 \eD^{6 \kappa }-49 \eD^{8 \kappa }+16 \eD^{10 \kappa }-6 \eD^{12 \kappa }+\eD^{14 \kappa }-1)\kappa ^3\\
&\quad\;\;\;-128 \eD^{4 \kappa } (8 \eD^{4 \kappa }-\eD^{2 \kappa }+8 \eD^{6 \kappa }-\eD^{8 \kappa }+2 \eD^{10 \kappa }+2)\kappa ^2\\
&\quad\;\;\;+(33 \eD^{2 \kappa }+312 \eD^{6 \kappa }+342 \eD^{8 \kappa }-342 \eD^{10 \kappa }-312 \eD^{12 \kappa }-33 \eD^{16 \kappa }-3 \eD^{18 \kappa }+3)\kappa\\
&\quad\;\;\;-6 (\eD^{2 \kappa }-1)^2 (\eD^{2 \kappa }+1)^3 (4 \eD^{2 \kappa }+10 \eD^{4 \kappa }+4 \eD^{6 \kappa }+\eD^{8 \kappa }+1)\Big)\,,
\end{align*}
and
\begin{align*}
&\Gamma_{{\rm II}}''(0)=\frac{\kappa\big(\eD^{-2 \kappa }-1\big)}{24\big(\eD^{2 \kappa }+1\big)^2}
\Big(16 \eD^{2 \kappa } (\eD^{4 \kappa }-\eD^{2 \kappa }+1)\kappa ^2-8 \eD^{2 \kappa } (\eD^{4 \kappa }-1)\kappa(\eD^{4 \kappa }-1)^2\Big)^{-1}\\
&\times\Big(32 \eD^{2 \kappa } (\eD^{2 \kappa }-2 \eD^{4 \kappa }+\eD^{6 \kappa }+\eD^{8 \kappa }+1)\kappa ^3-48 \eD^{4 \kappa } (\eD^{4 \kappa }-1)\kappa ^2
+6 (\eD^{4 \kappa }-1)^2 (\eD^{4 \kappa }+1)\kappa-(\eD^{4 \kappa }-1)^3\Big)^{-2}\\
&\times\Big(524288 \eD^{14 \kappa } (\eD^{2 \kappa }+1) {(\eD^{4 \kappa }-\eD^{2 \kappa }+1)}^2\kappa ^{10}\\
&\qquad
-16384 \eD^{6 \kappa } \Big(6 \eD^{2 \kappa }+4 \eD^{4 \kappa }-35 \eD^{6 \kappa }
-19 \eD^{8 \kappa }+148 \eD^{10 \kappa }-247 \eD^{12 \kappa }+247 \eD^{14 \kappa }\\
&\qquad\qquad\qquad\qquad
-148 \eD^{16 \kappa }+19 \eD^{18 \kappa }+35 \eD^{20 \kappa }-4 \eD^{22 \kappa }
-6 \eD^{24 \kappa }+3 \eD^{26 \kappa }-3\Big)\kappa ^9\\
&\qquad
+16384 \eD^{6 \kappa }\Big(98 \eD^{6 \kappa }-63 \eD^{4 \kappa }-\eD^{2 \kappa }+126 \eD^{8 \kappa }-265 \eD^{10 \kappa }+121 \eD^{12 \kappa }+121 \eD^{14 \kappa }-265 \eD^{16 \kappa }\\
&\qquad\qquad\qquad\qquad
+126 \eD^{18 \kappa }+98 \eD^{20 \kappa }-63 \eD^{22 \kappa }-\eD^{24 \kappa }+8 \eD^{26 \kappa }+8\Big)\kappa ^8\\
&\qquad
-1024\eD^{4\kappa}(\eD^{4\kappa} - 1) \Big(57 \eD^{2 \kappa }+339 \eD^{4 \kappa }-429 \eD^{6 \kappa }+2103 \eD^{8 \kappa }+2229 \eD^{10 \kappa }+1346 \eD^{12 \kappa }+1346 \eD^{14 \kappa }\\
&\qquad\qquad\qquad\qquad
+2229 \eD^{16 \kappa }+2103 \eD^{18 \kappa }-429 \eD^{20 \kappa }+339 \eD^{22 \kappa }+57 \eD^{24 \kappa }+19 \eD^{26 \kappa }+19\Big)\kappa ^7\\
&\qquad-
1024 \eD^{4 \kappa } {(\eD^{2 \kappa }-1)}^2 {(\eD^{2 \kappa }+1)}^3 \Big(34 \eD^{2 \kappa }+436 \eD^{4 \kappa }+1358 \eD^{6 \kappa }-1113 \eD^{8 \kappa }+792 \eD^{10 \kappa }\\
&\qquad\qquad\qquad\qquad
-1113 \eD^{12 \kappa }+1358 \eD^{14 \kappa }+436 \eD^{16 \kappa }+34 \eD^{18 \kappa }-67 \eD^{20 \kappa }-67\Big)\kappa ^6\\
&\qquad
-64 \eD^{2 \kappa } {(\eD^{2 \kappa }-1)}^3 {(\eD^{2 \kappa }+1)}^4 \Big(1047 \eD^{2 \kappa }+5045 \eD^{4 \kappa }+3676 \eD^{6 \kappa }-3990 \eD^{8 \kappa }-19046 \eD^{10 \kappa }\\
&\qquad\qquad\qquad\qquad
-3990 \eD^{12 \kappa }+3676 \eD^{14 \kappa }+5045 \eD^{16 \kappa }+1047 \eD^{18 \kappa }+49 \eD^{20 \kappa }+49\Big)\kappa ^5\\
&\qquad
+64\eD^{2\kappa}(\eD^{4\kappa} - 1)^4\Big(225 \eD^{2 \kappa }+913 \eD^{4 \kappa }+483 \eD^{6 \kappa }+12913 \eD^{8 \kappa }+12913 \eD^{10 \kappa }+483 \eD^{12 \kappa }\\
&\qquad\qquad\qquad\qquad
+913 \eD^{14 \kappa }+225 \eD^{16 \kappa }+130 \eD^{18 \kappa }+130\Big)\kappa ^4\\
&\qquad
-4(\eD^{4\kappa} - 1)^5\Big(1537 \eD^{2 \kappa }+13444 \eD^{4 \kappa }
+42404 \eD^{6 \kappa }+1838 \eD^{8 \kappa }+1838 \eD^{10 \kappa }+42404 \eD^{12 \kappa }\\
&\qquad\qquad\qquad\qquad
+13444 \eD^{14 \kappa }+1537 \eD^{16 \kappa }+9 \eD^{18 \kappa }+9\Big)\kappa ^3\\
&\qquad
-12(\eD^{4\kappa}- 1)^6\Big(499 \eD^{2 \kappa }+1389 \eD^{4 \kappa }-1327 \eD^{6 \kappa }-1327 \eD^{8 \kappa }+1389 \eD^{10 \kappa }+499 \eD^{12 \kappa }
-9 \eD^{14 \kappa }-9\Big)\kappa ^2\\
&\qquad
-3 (\eD^{4\kappa}- 1)^7 \Big(159 \eD^{2 \kappa }-1010 \eD^{4 \kappa }-1010 \eD^{6 \kappa }+159 \eD^{8 \kappa }+155 \eD^{10 \kappa }+155\Big)\kappa\\
&\qquad+24 (\eD^{4\kappa}- 1)^8 \Big(11 \eD^{2 \kappa }+11 \eD^{4 \kappa }+3 \eD^{6 \kappa }+3\Big)\Big).
\end{align*}

\end{appendix}


\newcommand{\etalchar}[1]{$^{#1}$}

\end{document}